\documentclass[10pt]{article}
\usepackage{newcent} 
\usepackage{amsmath,amssymb,amsthm,url,color}
\usepackage{enumerate}
\usepackage{hyperref}
\usepackage{graphicx}
\usepackage{epstopdf}
\def\today{\number\day\ \ifcase\month\or
	January\or February\or March\or April\or May\or June\or
	July\or August\or September\or October\or November\or December\fi
	\space \number\year}
\def\gobble#1#2{}
\def\shortdate{\number\day/\number\month/\expandafter\gobble\number\year}

\def\R{\mathbb{R}}
\def\p{Pain\-lev\'e}

\def\peqs{\p\ equations}

\def\det{\mathop{\rm det}\nolimits}

\def\ds{\displaystyle}

\newcommand{\HyperpFq}[2]{{}_{#1}F_{#2}}

\def\d{{\rm d}}\def\e{{\rm e}}

\def\a{\alpha}
\def\b{\beta}

\def\la{\lambda}

\def\w{\omega}
\def\ep{\varepsilon}

\def\ph{\varphi}
\def\vph{\varphi}
\def\A{\mathcal{A}}
\def\B{\mathcal{B}}

\newcommand{\Wr}{\operatorname{Wr}}

\newcommand{\deriv}[3][]{\frac{\d^{#1}{#2}}{{\d{#3}}^{#1}}}

\newtheorem{theorem}{Theorem}[section]
\newtheorem{proposition}[theorem]{Proposition}

\theoremstyle{definition}

\newtheorem{remark}[theorem]{Remark}
\newtheorem{remarks}[theorem]{Remarks}

\numberwithin{figure}{section}
\numberwithin{equation}{section}
\numberwithin{table}{section}

\newcommand{\comment}[1]{}

\def\la{\lambda}

\def\imp{\int_{-\infty}^{\infty}}

\def\beq{\begin{equation}}
	\def\eeq{\end{equation}}
\def\ds{\displaystyle}

\definecolor{dkg}{rgb}{0,0.7,0}
\definecolor{dkr}{rgb}{0.9,0,0}
\definecolor{dkb}{rgb}{0,0,0.7}
\definecolor{purple}{rgb}{0.5,0,0.7}

\definecolor{gold}{rgb}{0.83, 0.69, 0.22}

\def\url#1{\href{#1}{#1}}

\begin{document}
	
	\title{Generalised higher-order Freud weights}
	
	\author{Peter A. Clarkson$^{1}$, Kerstin Jordaan$^{2}$ and Ana Loureiro$^{1}$\\[2.5pt]
$^{1}$ School of Mathematics, Statistics and Actuarial Science,\\ University of Kent, Canterbury, CT2 7FS, UK\\[2.5pt]
$^{2}$ Department of Decision Sciences,\\
University of South Africa, Pretoria, 0003, South Africa\\[2.5pt]
Email: {P.A.Clarkson@kent.ac.uk};\enskip {jordakh@unisa.ac.za};\enskip{A.Loureiro@kent.ac.uk}
}
	
	\maketitle
	
\begin{abstract}
We discuss polynomials orthogonal with respect to a semi-classical generalised higher order Freud weight 
\[\omega(x;t,\lambda)=|x|^{2\lambda+1}\exp\left(tx^2-x^{2m}\right),\qquad x\in\mathbb{R},\]
with parameters $\lambda > -1$,  $t\in\mathbb{R}$ and $m=2,3,\dots$\ . The sequence of generalised higher order Freud weights for $m=2,3,\dots$, forms a hierarchy of weights, with associated hierarchies for the first moment and the recurrence coefficient. We prove that the first moment can be written as a finite partition sum of generalised hypergeometric $_1F_m$ functions and show that the recurrence coefficients satisfy difference equations which are members of the first discrete Painlev\'e hierarchy. We analyse the asymptotic behaviour of the recurrence coefficients and the limiting distribution of the zeros as $n \to \infty$. We also investigate structure and other mixed recurrence relations satisfied by the polynomials and related properties.  
	\end{abstract}

	\section{Introduction}
	In this paper we consider polynomials orthogonal with respect to the \textit{generalised higher order Freud weight}
\beq \label{freudg}
\w(x;t,\la)=|x|^{2\la+1}\exp\left(tx^2-x^{2m}\right),
\qquad x,t \in\R, \qquad m=2,3,\dots,\eeq 
with $\la>-1$ a parameter. 
{
The main goal of this paper is to bring a comprehensive self-contained analysis of these polynomials when the parameter \(m\) takes integer values higher than \(1\) and for any values of \(\lambda>-1\) and \(t\in\mathbb{R}\). The analysis for the particular cases of \(m=2,3\) was considered in \cite{refCJ18,refCJ21a,refCJ21b,refCJK,refMagnus85,refMagnus86}, with an emphasis on the study of the corresponding recurrence coefficients. We significantly extend existing studies on Freud type weights whilst providing a coherent and consistent approach, using techniques which are also likely to be adopted in the study of other semi-classical type weights. Throughout, we link and explain the connections to existing theory. After giving a short mathematical background in \S2, in \S \ref{sec:Freud4weight} we give a closed form expression for the moments with respect to the weight \eqref{freudg}, which correspond to a finite partition sum of generalised hypergeometric $_1F_m$ functions. 
The corresponding recurrence coefficients in the three term recurrence relation are investigated in \S \ref{sec:rcoef}. Therein we prove a recursive method that gives nonlinear recurrence relations satisfied by these  recurrence coefficients (see Proposition \ref{lm34}) and give them explicitly for the cases where \(m=4,5\), whilst recovering the already known relations for \(m=2,3\). We prove structure relations and mixed recurrence relations satisfied by generalised higher order Freud polynomials in \S \ref{sec:relations}. 
The asymptotic behaviour of the recurrence coefficients proved in \S \ref{sec:rcoef} determines the limiting distribution of the zeros and this, as well as other properties of the zeros, is investigated in \S \ref{sec:zeros}. We conclude with the quadratic decomposition of the generalised higher order Freud weight in \S \ref{sec:qdecomp}. }

\section{Mathematical background}

{Let \(\nu\) be a positive measure on the real line for which the support is not finite and all the moments 
\beq\label{eq:moment}
\mu_k=\imp x^k\,\d\nu (x), \qquad k=0,1,\dots,
\eeq
exist. The corresponding monic orthogonal polynomial sequence ${\big\{P_{n}(x)\big\}_{n\geq0}}$ is defined by \[\imp P_m(x)P_{n}(x)\,\d\nu (x) = h_{n}\delta_{m,n},\qquad h_{n}>0\]
where $\delta_{m,n}$ denotes the Kronekar delta. A fundamental property of orthogonal polynomials is that they satisfy a three-term recurrence relation of the form
	\beq \label{eq:3trr}
	P_{n+1}(x)=(x-\a_{n})P_{n}(x)-\b_{n}P_{n-1}(x),
	\eeq
	with $\b_{n}>0$ and initial values $P_{-1}(x)=0$ and $P_{0}(x)=1$. The recurrence coefficients $\a_{n}$ and $\b_{n}$ are given by the integrals
	\[
	\a_{n} = \frac1{h_{n}}\imp xP_{n}^2(x)\,\d\nu (x),\qquad \b_{n} = \frac1{h_{n-1}}\imp xP_{n-1}(x)P_{n}(x)\,\d\nu (x).
	\]
 Relevant for this paper is the case of a measure that admits a representation via a positive weight function \(\w(x)\) on the real line as follows \(\d \nu(x) = \w(x) \,\d x\). Henceforth, we will only work with a weight function representation. 
}

\comment{\color{red}	Let $\w(x)$ be a positive weight function defined on the real line $\R$ for which all the moments \beq \label{eq:moment}
	\mu_k=\imp x^k\w(x)\,\d x, \qquad k=0,1,\dots,\eeq exist. Then the sequence of monic orthogonal polynomials ${\big\{P_{n}(x)\big\}_{n \geq0}}$, where $P_{n}(x)$ is a polynomial of degree $n$ in $x$, is given by \beq \nonumber\imp P_m(x)P_{n}(x)\,\w(x)\,\d x = h_{n}\delta_{m,n},\qquad h_{n}>0,\label{eq:norm}\eeq
	where $\delta_{m,n}$ denotes the Kronekar delta.
	 
A fundamental property of orthogonal polynomials is that they satisfy a three-term recurrence relation of the form
	\beq \label{eq:3trr}
	P_{n+1}(x)=(x-\a_{n})P_{n}(x)-\b_{n}P_{n-1}(x),
	\eeq
	with $\b_{n}>0$ and initial values $P_{-1}(x)=0$ and $P_{0}(x)=1$. The recurrence coefficients $\a_{n}$ and $\b_{n}$ are given by the integrals
	\[
	\a_{n} = \frac1{h_{n}}\imp xP_{n}^2(x)\,\w(x)\,\d x,\qquad \b_{n} = \frac1{h_{n-1}}\imp xP_{n-1}(x)P_{n}(x)\,\w(x)\,\d x.
	\]
 }
 
The coefficient $\b_{n}$ in the recurrence relation \eqref{eq:3trr} can also be expressed in terms of the Hankel determinant 	\beq\label{eq:detsDn}
	\Delta_{n}=\det\big[\mu_{j+k}\big]_{j,k=0}^{n-1}=\left|\begin{matrix} \mu_0 & \mu_1 & \ldots & \mu_{n-1}\\
		\mu_1 & \mu_2 & \ldots & \mu_{n}\\
		\vdots & \vdots & \ddots & \vdots \\
		\mu_{n-1} & \mu_{n} & \ldots & \mu_{2n-2}\end{matrix}\right|,\qquad n\geq1,\eeq
	with $\Delta_0=1$, $\Delta_{-1}=0$, whose entries are given in terms of the moments \eqref{eq:moment} associated with the weight $\w(x)$. Specifically
	\beq\label{def:bn}
	\b_{n} = \frac{\Delta_{n+1}\Delta_{n-1}}{\Delta_{n}^2}.\eeq
	
	The monic polynomial $P_{n}(x)$ can be uniquely expressed as the determinant
	\beq\nonumber P_{n}(x)=\frac1{\Delta_{n}}\left|\begin{matrix} \mu_0 & \mu_1 & \ldots & \mu_{n}\\
		\mu_1 & \mu_2 & \ldots & \mu_{n+1}\\
		\vdots & \vdots & \ddots & \vdots \\
		\mu_{n-1} & \mu_{n} & \ldots & \mu_{2n-1}\\
		1 & x & \ldots &x^n\end{matrix}\right|,\label{eq:Pndet}\eeq
	and the normalisation constants as
	\beq\label{def:norm}h_{n}=\frac{\Delta_{n+1}}{\Delta_{n}},\qquad h_0=\Delta_1=\mu_0.
	\eeq
	Also from \eqref{def:bn} and \eqref{def:norm}, we see that the relationship between the recurrence coefficient $\b_{n}$ and the normalisation constants $h_{n}$ is given by
	\beq\label{bn:hn}\nonumber h_{n}=\b_{n} h_{n-1}.\eeq 

For symmetric weights, since $\w(x)=\w(-x)$, it follows that $\a_{n}=0$ in \eqref{eq:3trr}. Hence, for symmetric weights, the sequence of monic orthogonal polynomials ${\big\{P_{n}(x)\big\}_{n\geq0}}$, satisfy the three-term recurrence relation
\beq\label{eq:srr}
P_{n+1}(x)=xP_{n}(x)-\b_{n}P_{n-1}(x).
\eeq
The monic orthogonal polynomials $P_{n}(x)$ associated with symmetric weights are also symmetric, i.e. $P_{n}(-x)=(-1)^nP_{n}(x)$. This implies that each $P_{n}$ contains only even or only odd powers of $x$ and we can write 
\begin{align*} P_{2n}(x)&=x^{2n}+\sum_{k=1}^{n}c_{2n-2k}^{(2n)} x^{2n-2k}
=x^{2n}+c_{2n-2}^{(2n)}x^{2n-2}+\dots+c_{0}^{(2n)},\\
	P_{2n+1}(x)&=x^{2n+1}+\sum_{k=1}^{n}c_{2n-2k+1}^{(2n+1)}x^{2n-2k+1}
	=x^{2n+1}+c_{2n-1}^{(2n+1)}x^{2n-1}+\dots+c_{1}^{(2n+1)}x.
\end{align*} Substituting these expressions into the recurrence relation \eqref{eq:srr} and comparing the coefficients on each side, we obtain
\begin{align}
		\b_{2n}&=c_{2n-2}^{(2n)}-c_{2n-1}^{(2n+1)},\qquad \label{bodd}
		\b_{2n+1}=-\frac{c_0^{(2n+2)}}{c_0^{(2n)}}=-\frac{P_{2n+2}(0)}{P_{2n}(0)}.\end{align}
It follows from \eqref{eq:moment} that, for symmetric weights, $\mu_{2k-1}=0$, $k=1,2,\ldots\ $ and hence 
it is possible to write the Hankel determinant $\Delta_{n}$ given by \eqref{eq:detsDn} in terms of the product of two Hankel determinants obtained by matrix manipulation, interchanging columns and rows. The product decomposition, depending on $n$ even or odd, is given by
	\beq \Delta_{2n}=\A_{n}\B_{n},\qquad \Delta_{2n+1}=\A_{n+1}\B_{n},\label{res:lemma21}\eeq
	where $\A_{n}$ and $\B_{n}$ are the Hankel determinants 
	\begin{align}\label{def:AnBn} \A_{n} &
		{=\left|\begin{matrix} 
\mu_0 & \mu_2 & \ldots & \mu_{2n-2}\\
\mu_2 & \mu_4 & \ldots & \mu_{2n} \\
\vdots & \vdots & \ddots & \vdots \\
\mu_{2n-2} & \mu_{2n}& \ldots & \mu_{4n-4}
		\end{matrix}\right|,}\qquad
		\B_{n} 
		{=\left|\begin{matrix} 
\mu_2 & \mu_4 & \ldots & \mu_{2n}\\
\mu_4 & \mu_6 & \ldots & \mu_{2n+2} \\
\vdots & \vdots & \ddots & \vdots \\
\mu_{2n} & \mu_{2n+2}& \ldots & \mu_{4n-2}
		\end{matrix}\right|},\end{align}
with $\A_0=\B_0=1$. Consequently, for a symmetric weight, substituting \eqref{res:lemma21} into \eqref{def:bn}, the recurrence coefficient $\b_{n}$ is given by
	\beq \nonumber\b_{2n} = \frac{\A_{n+1}\B_{n-1}}{\A_{n}\B_{n}},\qquad
	\b_{2n+1}= \frac{\A_{n}\B_{n+1}}{\A_{n+1}\B_{n}}.\label{def:betan}
	\eeq

Semiclassical orthogonal polynomials are natural generalisations of classical orthogonal polynomials and were introduced by Shohat in \cite{refShohat39}. Maroni provided a unified theory for semiclassical orthogonal polynomials {(cf.~\cite{refMaroni,refMaroni2})}. The weights of classical orthogonal polynomials satisfy a first-order ordinary differential equation, the \textit{Pearson equation} 
	\begin{equation}\label{eq:Pearson}
		\deriv{}{x}\{\sigma(x)\,\w(x)\}=\tau(x)\,\w(x),
	\end{equation}
	where $\sigma(x)$ is a monic polynomial of degree at most $2$ and $\tau(x)$ is a polynomial with degree $1$. For semiclassical orthogonal polynomials, the weight function $\w(x)$ satisfies a Pearson equation 
	\eqref{eq:Pearson} 
	with either deg$(\sigma(x))>2$ or deg$(\tau(x))\neq1$ (cf.~\cite{refHvR,refMaroni}). 
The generalised higher order Freud weight given by \eqref{freudg}
	is a symmetric weight that satisfies the Pearson equation \eqref{eq:Pearson}
 with $\sigma(x)=x$ and $\tau(x)=2(tx^2-mx^{2m}+\la+1)$ and therefore is a semiclassical weight.
 
 \comment{\color{purple}\marginpar{remove??} 
 In \S \ref{sec:Freud4weight} we consider the moments of the generalised higher order Freud weight, obtaining a closed form expression for the first moment. The recurrence coefficients in the three term recurrence relation satisfied by polynomials orthogonal with respect to generalised higher order Freud weights are investigated in \S \ref{sec:rcoef}. We prove structure relations and mixed recurrence relations satisfied by generalised higher order Freud polynomials in \S \ref{sec:relations}. The asymptotic behaviour of the recurrence coefficients proved in \S \ref{sec:rcoef} determines the limiting distribution of the zeros and this, as well as other properties of the zeros, is investigated in \S \ref{sec:zeros}. We conclude with the quadratic decomposition of the generalised higher order Freud weight in \S \ref{sec:qdecomp}. }

	\section{Moments of the generalised higher order Freud weights}\label{sec:Freud4weight}
		
		The existence of the first moment $\mu_0(t;\la,m)$ associated with the generalised higher order Freud weight \eqref{freudg} follows from the fact that, at $\infty$, the integrand behaves like $\exp(-x^2)$ and, at $x=0$, the integrand behaves like $x^{\la}$ which, for $\la>-1$, is integrable. 

	\begin{theorem} \label{hypex}
		Let $x\in \R, ~\la>-1,~ t\in \R$ and $m=2,3,\dots$. Then, for the generalised higher order Freud weight \eqref{freudg},
		the first moment is given by
		 \begin{align*}\mu_0(t;\la,m) &=\int_{-\infty}^{\infty} |x|^{2\la+1}\exp(tx^2-x^{2m})\,\d x = \int_0^\infty s^{\la}\exp(ts-s^m)\, \d s \nonumber\\
&= \ds \frac{1}{m}\sum_{k=1}^m\frac{t^{k-1}}{{(k-1)}!} \Gamma\left(\frac{\la+ k}{m}\right) \;\HyperpFq{2}{m}\left(\frac{\la+ k}{m},1;\frac km,\frac{k+1}{m},\dots,\frac{m+k-1}{m};\left(\frac{t}{m}\right)^m\right)\\
	&=\ds \frac{1}{m} \Gamma\left(\frac{\la+1}{m}\right) \;\HyperpFq{1}{m-1}\left(\frac{\la+1}{m};\frac 1m,\dots\frac{m-1}{m};\left(\frac{t}{m}\right)^m\right)\\
&\qquad+\ds \frac{1}{m} \sum_{k=2}^{m-1}\frac{t^{k-1}}{{(k-1)}!} \Gamma\left(\frac{\la+ k}{m}\right)\nonumber\\ &\qquad\qquad\times\HyperpFq{1}{m-1}\left(\frac{\la+ k}{m};\frac km,\frac{k+1}{m},\dots,\frac{m-1}{m},\frac{m+1}{m},\dots,\frac{m+k-1}{m};\left(\frac{t}{m}\right)^m\right)\nonumber \\
&\qquad+\ds \frac{t^{m-1}}{m!} \Gamma\left(\frac{\la}{m}+1\right) \;\HyperpFq{1}{m}\left(\frac{\la}{m}+1;\frac{m+1}{m},\frac{m+2}{m},\dots,\frac{2m-1}{m};\left(\frac{t}{m}\right)^m\right),\nonumber
\end{align*}
		where $\HyperpFq pq(a;b;z)$ is the generalised hypergeometric function (cf.~\cite[eq.\ 16.2.1]{DLMF}. 
	\end{theorem}\begin{proof}
		Using the power series expansion of the exponential function, we obtain
		\[ \begin{split}\mu_0(t;\la,m) &=\int_{-\infty}^{\infty} |x|^{2\la+1}\exp(tx^2-x^{2m})\,\d x = \int_0^\infty s^{\la}\exp(ts-s^m)\, \d s \\
&= \ds \int_0^\infty s^{\la}\exp(-s^m)\sum_{n=0}^{\infty}\frac{(ts)^n}{n!}\, \d s\\
&= \ds \sum_{n=0}^{\infty}\frac{t^n}{n!}\int_0^\infty s^{n+\la}\exp(-s^m)\, \d s\\
&= \ds\frac 1m \sum_{n=0}^{\infty}\frac{t^n}{n!} \int_0^\infty y^{(n+\la-m+1)/m}\exp(-y)\, \d y\\
&= \ds \frac 1m\sum_{n=0}^{\infty} \frac{t^n}{n!}\,\Gamma \left(\frac{\la+n+1}{m}\right),
		\end{split}\] where $\Gamma(x)$ denotes the Gamma function defined in \cite[eq.\ 5.2.1]{DLMF} and the fourth equal sign follows from the Lebesgue's Dominated Convergence Theorem. 
		Letting $n=mk+j$ for $j=0,1,\dots,m-1$, we can write 
		\[ \mu_0(t;\la,m) =\ds \frac 1m\sum_{k=0}^{\infty}\sum_{j=0}^{m-1}\Gamma\left(\frac{\la+j+1}{m}+k\right)\;\frac{t^{mk+j}}{(mk+j)!}.\] Using the Gauss multiplication formula \cite[eq.\ 5.5.6]{DLMF} yields 
		\[(mk+j)!=j!\,m^{mk} \prod_{\ell=1}^{m}\left(\frac{j+\ell}{m}\right)_k\]
	where $(a)_k$ denotes the Pochhammer symbol (cf.~\cite[\S5.2(iii)]{DLMF},	while it follows from \cite[eq.\ 5.5.1]{DLMF} that
		\[\Gamma\left(\frac{\la+j+1}{m}+k\right)=\left(\frac{\la+j+1}{m}\right)_k\Gamma\left(\frac{\la+j+1}{m}\right),\] 
		and hence we have
		\[\begin{split} \mu_0(t;\la,m) &=\ds \frac 1m\sum_{k=0}^{\infty}\sum_{j=0}^{m-1}\frac{\left(\frac{\la+j+1}{m}\right)_k\Gamma\left(\frac{\la+j+1}{m}\right)}{m^{mk}\left(\frac{j+1}{m}\right)_k\left(\frac{j+2}{m}\right)_k...\left(\frac{j+m}{m}\right)_k}\,\frac{t^{mk+j}}{j!}\\
&=\ds \frac 1m\sum_{j=0}^{m-1}\Gamma\left(\frac{\la+j+1}{m}\right)\frac{t^j}{j!}\sum_{k=0}^{\infty}\frac{\left(\frac{\la+j+1}{m}\right)_k}{\left(\frac{j+1}{m}\right)_k\left(\frac{j+2}{m}\right)_k...\left(\frac{j+m}{m}\right)_k}\left(\frac{t}{m}\right)^{mk}\\
&= \ds \frac 1m\sum_{j=0}^{m-1}\Gamma\left(\frac{\la+j+1}{m}\right)\frac{t^j}{j!}\;\HyperpFq{2}{m}\left(\frac{\la+j+1}{m},1;\frac{j+1}{m},\frac{j+2}{m},\dots,\frac{m+j}{m};\left(\frac{t}{m}\right)^m\right),
		\end{split}\]
as required.
	\end{proof}
\begin{remark}
	In our earlier studies of semi-classical orthogonal polynomials, we proved special cases of Theorem \ref{hypex} and Theorem \ref{momentDE}, namely for $m=2$ in \cite{refCJK} and for $m=3,4,5$ in \cite{refCJ21b}. 
\end{remark}
{In the following theorem we derive a differential equation satisfied by the first moment $\mu_0(t;\la,m)$. It is often much easier to derive properties of a function from the  differential equation it satisfies rather than from an integral representation or, as this case, a sum of generalised hypergeometric functions.}
	\begin{theorem}\label{momentDE} Let $x\in \R, ~\la>-1,~ t\in \R$ and $m=2,3,\dots$\ . Then, for the generalised higher order Freud weight \eqref{freudg},
		the first moment 
		\[\mu_0(t;\la,m) =\int_{-\infty}^{\infty} |x|^{2\la+1}\exp(tx^2-x^{2m})\,\d x = \int_0^\infty s^{\la}\exp(ts-s^m)\, \d s, \]
		satisfies the ordinary differential equation
		\beq\label{eq6} m\deriv[m]{\vph}{t} - t \deriv{\vph}{t} - (\la+1)\,\vph=0. \eeq
	\end{theorem}
	\begin{proof}
		Following \cite{refMul77} and \cite{refCJ21b}, we look for a solution of \eqref{eq6} in the form
		\beq \label{eq7}\vph(t)=\int_0^\infty \e^{st} \,v(s)\,\d s.\eeq
		In order for \eqref{eq7} to satisfy \eqref{eq6}, it is necessary that
		\[\deriv[m]{\vph}{t}-{\frac{t}m}\deriv{\vph}{t}-\frac{\la+1}{m}\vph=\int_0^\infty \e^{st} \left(s^m-{\frac{ts}{m}}-\frac{\la+1}{m}\right)v(s) \,\d s=0.\]
		Using integration by parts, this is equivalent to
		\[\int_0^\infty \e^{st} \left\{ s^mv(s) +\frac1m v(s)+\frac{s}m \deriv{v}{s}-\frac{\la+1}{m}v(s)\right\}\d s=0,\] 
		under the assumption that $\lim_{s\to\infty} sv(s)\e^{st}=0$.
		Hence, for $\vph(t)$ to be a solution of \eqref{eq6}, we need to choose $v(s)$ so that 
		\[	 (ms^m-\la)v(s)+s\deriv{v}{s}=0 .\] One solution of this equation is $v(s) = s^{\la}\exp(-s^m)$.\end{proof}

For the generalised higher order Freud weight \eqref{freudg}, the even moments can be written in terms of derivatives of the first moment, as follows
\begin{align} \label{momd}\mu_{2k}(t;\la,m)&= \imp x^{2k} |x|^{2\la+1} \exp(tx^2-x^{2m})\,\d x \nonumber\\
	&= \deriv[k]{}{t} \imp |x|^{2\la+1}\exp(tx^2-x^{2m})\,\d x\nonumber \\&=\deriv[k]{}{t}\mu_0(t;\la,m),\qquad k=0,1,2,\ldots\ ,\end{align}
where the interchange of integration and differentiation is justified by Lebesgue's Dominated Convergence Theorem.
Furthermore, from the definition we have 
\begin{align} \label{momd2}\mu_{2k+2}(t;\la,m)&= 
\mu_{2k}(t;\la+1,m),\qquad k=0,1,2,\ldots\ , \end{align}
{and therefore
\[ \mu_{2k+2}(t;\la,m) = \mu_0(t;\la+k+1,m),\qquad k=0,1,2,\ldots\ , \]
which illustrates the importance of the first moment.}

\section{Recurrence coefficients for generalised higher order Freud weights}\label{sec:rcoef}

\begin{theorem}For the generalised higher order Freud weight \eqref{freudg}, the recurrence coefficient $\b_{n}$ is given by
	\beq \b_{2n} = \deriv{}{t} \ln \frac{\B_{n}}{\A_{n}},\qquad
	\b_{2n+1}= \deriv{}{t} \ln\frac{\A_{n+1}}{\B_{n}}. \label{def:betant}
	\eeq with $A_0=B_0=1$ and \beq \A_{n} =\Wr\left(\mu_0,\deriv{\mu_0}{t},\ldots,\deriv[n-1]{\mu_0}{t} \right), 
	\qquad \B_{n} =\Wr\left(\deriv{\mu_0}{t},\deriv[2]{\mu_0}{t},\ldots,\deriv[n]{\mu_0}{t} \right),
	\label{def:AnBnW}\eeq
	 where
	\[\mu_0= \mu_0(t;\la,m) =\int_{0}^{\infty} x^{\la}\exp(ts-s^m)\,\d{x},\]
	and $\Wr(\ph_1,\ph_2,\ldots,\ph_{n})$ denotes the Wronskian given by
	$$\Wr(\ph_1,\ph_2,\ldots,\ph_{n})= \left|\begin{matrix} 
		\ph_1 & \ph_2 & \ldots & \ph_{n}\\
		\ph_1^{(1)} & \ph_2^{(1)} & \ldots & \ph_{n}^{(1)}\\
		\vdots & \vdots & \ddots & \vdots \\
		\ph_1^{(n-1)} & \ph_2^{(n-1)} & \ldots & \ph_{n}^{(n-1)}
	\end{matrix}\right|,\qquad \ph_j^{(k)}=\deriv[k]{\ph_j}{t}.$$
	 
\end{theorem}
\begin{proof} It follows from \eqref{def:AnBn} and \eqref{momd} that $\A_{n}$ and $\B_{n}$ can be written in terms of the Wronskians given by \eqref{def:AnBnW}. Furthermore, \begin{align}\label{eq:dodgson} &\A_{n}\deriv{\B_{n}}{t}-\B_{n}\deriv{\A_{n}}{t}=\A_{n+1}\B_{n-1},\qquad
		\B_{n}\deriv{\A_{n+1}}{t}-\A_{n+1}\deriv{\B_{n}}{t}=\A_{n+1}\B_{n}\end{align} (cf.~\cite[\S6.5.1]{refVeinDale}) and \eqref{eq:dodgson}, together with \eqref{def:betan} yields \eqref{def:betant}.
\end{proof}

\begin{theorem} Let $\w_0(x)$ be a symmetric positive weight on the real line for which all the moments exist and let $\w(x;t)=\exp(tx^2)\,\w_0(x)$, with $t\in\R$, is a weight such that all the moments of exist. Then the recurrence coefficient $\b_{n}(t)$ satisfies the Volterra, or the Langmuir lattice, equation
	\beq \nonumber\deriv{\b_{n}}{t} = \b_{n}(\b_{n+1}-\b_{n-1}). \eeq
\end{theorem}

\begin{proof} {See, for example, Van Assche \cite[Theorem 2.4]{refWVAbk}.}
\end{proof}

\begin{theorem} For the generalised higher order Freud weight \eqref{freudg}, the associated monic polynomials $P_{n}(x)$ satisfy the recurrence relation 
	\beq \label{eq:3rr} P_{n+1}(x) = xP_{n}(x) - \b_{n}(t;\la) P_{n-1}(x),\qquad n=0,1,2,\ldots\ ,\eeq
	with $P_{-1}(x)=0$ and $P_0(x)=1$, where
 \begin{align*} \b_{2n}(t;\la) &
= \frac{\A_{n+1}(t;\la)\A_{n-1}(t;\la+1)}{\A_{n}(t;\la)\A_{n}(t;\la+1)}=\deriv{}{t}\ln \frac{\A_{n}(t;\la+1)}{\A_{n}(t;\la)},\\
\b_{2n+1}(t;\la)&
= \frac{\A_{n}(t;\la)\A_{n+1}(t;\la+1)}{\A_{n+1}(t;\la)\A_{n}(t;\la+1)}=\deriv{}{t}\ln \frac{\A_{n+1}(t;\la)}{\A_{n}(t;\la+1)}. 
	\end{align*} 
	where $\A_{n}(t;\la)$ is the Wronskian given by \eqref{def:AnBnW}
	with 
	\begin{align*}
	\mu_0(t;\la,m) 
&= \ds \frac{1}{m}\sum_{k=1}^m\frac{t^{k-1}}{{(k-1)}!} \Gamma\left(\frac{\la+ k}{m}\right) \;\HyperpFq{2}{m}\left(\frac{\la+ k}{m},1;\frac km,\frac{k+1}{m},\dots,\frac{m+k-1}{m};\left(\frac{t}{m}\right)^m\right).
	\end{align*}
\end{theorem}
\begin{proof}
	It follows from substituting \eqref{momd2} into the expression for $\B_{n}(t;\la)$ given in \eqref{def:AnBnW} that $\B_{n}=\A_{n}(t;\la+1)$ and then the result immediately follows from \eqref{def:betan} and \eqref{def:betant}. \end{proof}

\subsection{Nonlinear recursive relations}
We follow the approach found in \cite[\S7]{refMaroni2} whose key results are summarised in \cite[Proposition 3.1]{Mar99}.

Note that for a given $m\geq 1$, we can write 
\begin{equation}\label{x2mPn}
	x^{2m} P_{n}(x) 
	= \sum_{\ell=-m}^{m} C^{(2m)}_{n,n+2 \ell} P_{n+2 \ell},
\end{equation}
where 
\[
	C^{(2m)}_{n,n+2 \ell} = \frac{1}{h_{n+2 \ell}} \int_{-\infty}^\infty x^{2m} P_{n+2 \ell}(x)P_{n}(x)\,\w(x)\,\d x
	\quad \text{for} \quad \ell = -m, \ldots, m. 
\]
Observe that 
\(
C^{(2m)}_{n,n+k} = C^{(2m)}_{n+k,n}=0 \) for \( |k|\geq 2m+1 \) 
and \[
 C^{(2m)}_{n,n+2\ell} =\frac{h_{n}}{h_{n+2\ell}}C^{(2m)}_{n+2\ell,n}
 = \frac{1}{\beta_{n+1}\cdots \beta_{n+2\ell}} C^{(2m)}_{n+2\ell,n} \, \quad \text{for} \quad \ell=1,\ldots, m . 
\]
From the recurrence relation \eqref{eq:srr} it follows 
\begin{equation}\label{x2Pn}
	x^2P_{n}(x) =
 P_ {n+2} + \left(\b_{n}+\b_ {n+1}\right) P_{n}+\b_ {n-1} \b_{n} P_ {n-2}, \quad n\geq 0. 
\end{equation}
 In particular, one has $C_{n,n}^{(2)}=\beta_{n}+\beta_{n+1}$, $C_{n,n-2}^{(2)}=\beta_{n-1}\beta_{n}$ and $C_{n,n+2}^{(2)} = 1$. The computation of the coefficients \( \big\{C_{n-2\ell,n}^{(2m+2)}\big\}_{\ell=0}^{m+1}\) can be derived from the coefficients \( \big\{C_{n-2\ell,n}^{(2m)}\big\}_{\ell=0}^{m}\) as follows 
\begin{equation}\label{Ckn2m}
	C_{n-2\ell,n}^{(2m+2)} 
	= \beta_{n+2}\beta_{n+1} C_{n-2\ell,n+2}^{(2m)} + (\beta_{n}+\beta_{n+1}) C_{n-2\ell,n}^{(2m)} 
	+ C_{n-2\ell,n-2}^{(2m)} , \quad \ell= 0, \ldots , m, 
\end{equation}
which is a direct consequence of \eqref{x2mPn} multiplied by \(x^2\) and \eqref{x2Pn}.

\begin{proposition}\label{lm34} The recurrence coefficient $\b_{n}$ for the generalised higher-order Freud weight \eqref{freudg}
satisfies the discrete equation
\begin{equation}\label{Vn2m}
	2m V_{n\phantom{1}}^{(2m)} - 2t \beta_{n} = n+ (\la +\tfrac12) [1-(-1)^n]. 
\end{equation}
where 
\begin{equation} \label{Vn2m exp}
	V_{n\phantom{1}}^{(2m)}= C_{n,n-2}^{(2m-2)}+\beta_{n}C_{n,n}^{(2m-2)}.
\end{equation}
\end{proposition}

Alternatively, \eqref{Vn2m exp} can be written as
\[
	V_{n\phantom{1}}^{(2m)}= 
 \frac{1}{h_{n-2}} \int_{-\infty}^\infty x^{2m-2} P_{n-2}(x)P_{n}(x)\,\w(x)\,\d x 
			+\frac{\beta_{n}}{h_{n}} \int_{-\infty}^\infty x^{2m-2} P^2_{n}(x)\,\w(x)\,\d x .
\]

\begin{proof}
For any monic polynomial sequence $\big\{P_{n}(x)\big\}_{n\geq0}$, one can always write 
\[
 x \deriv{P_{n}}{x}(x)= \sum_{j=0}^{n} \rho_{n,j} P_{n-j}(x), \quad \text{for} \quad n\geq 1, 
\]
with $\rho_{n,0}=n$.
The assumption that $\big\{P_{n}(x)\big\}_{n\geq0}$ is orthogonal with respect to the semiclassical weight $\w(x)$ satisfying the differential equation \eqref{eq:Pearson} with $\sigma(x)=x$ and $\tau(x)=2(tx^2-mx^{2m}+\la+1)$ gives, using integration by parts, 
\begin{align*}
 \rho_{n,j} h_{n-j}&= \int_{-\infty}^\infty x \deriv{P_{n}}{x}(x) P_{n-j}(x)\,\w(x)\,\d x\\
 &= - \int_{-\infty}^{\infty} \left\{ \tau(x)P_{n-j}(x) + x \deriv{P_{n-j}}{x}(x) \right\}
 P_{n}(x)\w(x)\,\d x,
\end{align*} where $\ds h_k=\int_{-\infty}^{\infty}P_k^2(x)\,\w(x)\,\d x>0.$
Therefore 
$\rho_{n,j} =0 $ for any $j\geq 2m+1$ and 
the 
symmetry of the weight implies \(\rho_{n,j}=0\) for any \(j\) odd. Therefore we have 
\begin{equation}\label{eq: struct Pn}
 x \deriv{P_{n}}{x}(x) = \sum_{\ell=0}^{m} \rho_{n,2\ell}\, P_{n-2\ell}(x), \quad \text{for} \quad n\geq 0.
\end{equation}
Recall \eqref{x2Pn} to write 
\[ 
 \frac{1}{h_{n}}\int_{-\infty}^\infty x^2 P^2_{n}(x)\,\w(x)\,\d x = (\beta_{n} + \beta_{n+1}) 
\quad \text{and} \quad 
 \frac{1}{h_{n-2}}\int_{-\infty}^\infty x^2 P_{n-2}(x)P_{n}(x)\,\w(x)\,\d x = \beta_{n} \beta_{n-1}, \]
and hence 
\begin{equation}\label{rhosys1}
\rho_{n,2\ell} = 
\begin{cases}
 \displaystyle 
 \frac{2m}{h_{n}} \int_{-\infty}^\infty x^{2m} P^2_{n}(x)\,\w(x)\,\d x 
 - 2t (\beta_{n} + \beta_{n+1})
 - \left(2\la+2 +n\right),
 & \text{if} \quad \ell=0, \\[0.25cm]
 \displaystyle \frac{2m}{h_{n-2}} \int_{-\infty}^\infty x^{2m} P_{n-2}(x) P_{n}(x)\,\w(x)\,\d x
 - 2t \, \beta_{n}\beta_{n-1},
 & \text{if} \quad \ell=1, \\[0.25cm]
 \displaystyle\frac{2m}{h_{n-2\ell}} \displaystyle \int_{-\infty}^\infty x^{2m} P_{n-2\ell}(x)P_{n}(x)\,\w(x)\,\d x,
 & \text{if} \quad 2 \leq \ell \leq m-1, \\[0.25cm]
 \displaystyle {2m} \, \beta_{n} \cdots \beta_{n-2m+1},
 & \text{if} \quad \ell=m, \\[0.25cm]
 0, & \text{otherwise}. 
\end{cases}
\end{equation}

Take \(\ell=0\) in \eqref{Ckn2m} and note that $C_{n,n-2}^{(2m-2)} = C_{n-2,n}^{(2m-2)} \beta_{n}\beta_{n-1}$ to get 
\[
	C_{n,n}^{(2m)} 
	= \beta_{n+1} \left( C_{n,n+2}^{(2m-2)} \beta_{n+2} + C_{n,n}^{(2m-2)} \right)
	+ \beta_{n} \left( C_{n-2,n}^{(2m-2)} \beta_{n-1} + C_{n,n}^{(2m-2)} \right) . 
\]
The symmetric orthogonality recurrence relation \eqref{eq:srr} implies that 
\[
	P_{n+2}(x)P_{n}(x) 
	= P_{n+1}^2(x)+ \beta_{n} P_{n-1}(x) P_{n+1}(x) - \beta_{n+1} P_{n}^2(x),
\]
which gives the relation 
\begin{equation} \label{rel cns 2m}
	C_{n,n+2}^{(2m-2)} \beta_{n+2} + C_{n,n}^{(2m-2)} 
	= C_{n-1,n+1}^{(2m-2)} \beta_{n} + C_{n+1,n+1}^{(2m-2)} , 
\end{equation}
and consequently we have 
\begin{equation} \label{cnn2m}
	C_{n,n}^{(2m)} 
	= V_{n+1}^{(2m)} + V_{n\phantom{1}}^{(2m)}
\quad \text{where}\quad 
	V_{n\phantom{1}}^{(2m)}=\beta_{n} \left( \beta_{n-1} C_{n-2,n}^{(2m-2)}+ C_{n,n}^{(2m-2)}\right).
\end{equation}
 
On the other hand, expressions for the coefficients \(\rho_{n,2j}\) can be obtained through a purely algebraic way, and therefore expressed recursively. For that, we differentiate with respect to $x$ the recurrence relation \eqref{eq:srr} and use the structure relation \eqref{eq: struct Pn} to get 
\[
 P_{n}(x) + \sum_{\ell=0}^{m} \rho_{n,2\ell} P_{n-2\ell}(x) 
 = \deriv{P_{n+1}}{x}(x) + \b_{n} \deriv{P_{n-1}}{x}(x). 
\]
We multiply the latter by $x$ and use again \eqref{eq: struct Pn} and then \eqref{eq:srr} to obtain a linear combination of terms of $\big\{P_{n}(x)\big\}_{n\geq0}$ and this gives 
\[
\begin{aligned}
 P_{n+1}(x) &+ \b_{n} P_{n-1}(x)
 = \sum_{\ell=0}^{m+1} \left( \rho_{n+1,2\ell}- \rho_{n,2\ell}
 +\b_{n} \, \rho_{n-1,2\ell-2}- \b_{n-2\ell+2}\,\rho_{n,2\ell-2}\right) P_{n-2\ell +1}(x). 
\end{aligned}
\]
Since the terms are linearly independent, we equate the coefficients of $P_{n+1}, P_{n}, \ldots , P_{n-2m-1}$ to get 
\begin{equation}\label{rhosys2}
\begin{cases}
 \rho_{n,0} = n, \\
 \rho_{n+1,2}- \rho_{n,2} = 2 \b_{n}, \\
 \rho_{n+1,2\ell}- \rho_{n,2\ell}
 = \b_{n-2\ell+2}\,\rho_{n,2\ell-2}- \b_{n} \, \rho_{n-1,2\ell-2}, 
 & \text{for}\quad \ell=2, \ldots , m-1,\\
 \b_{n-2m}\,\rho_{n,2m} = \b_{n} \,\rho_{n-1,2m}, &\text{for}\quad j=m-1. 
\end{cases}
\end{equation}

We combine \eqref{rhosys1} with \eqref{rhosys2} to conclude that the first equation (when \(\ell=0\)) gives 
\begin{equation*}
	m V_{n+1}^{(2m)} + 	m V_{n\phantom{1}}^{(2m)} 
 - t (\beta_{n} + \beta_{n+1})
 = n + \left(\la+1 \right) , 
\end{equation*} 
which implies \eqref{Vn2m}. 
\end{proof} 

The expressions for $V_{n\phantom{1}}^{(2m)}$ can be then obtained recursively using \eqref{Vn2m exp}, \eqref{Ckn2m} 
and \eqref{cnn2m} to write
\begin{align}
 V_{n\phantom{1}}^{(2m)} & = \beta_{n} \left( V^{(2 m -2)}_ {n+1}+V^{(2 m -2)}_{n}\right)
 +\left( \beta_{n}+\beta_ {n+1}\right) V^{(2 m -2)}_{n}\nonumber\\
 &\qquad
 -\beta_{n} \left(\beta_{n}+\beta_ {n+1}\right)
 \left( V_{n\phantom{1}}^{(2 m -4)}+ V_{ n +1}^{(2 m -4)}\right)
 +\beta_{n} \beta_ {n-1} \left( V_ {n-2}^{(2 m -4)}
 + V_{ n -1}^{(2 m -4)}\right)\nonumber\\
 &\qquad
 + \beta_{n} \beta_ {n-1} \beta_ {n+1} \beta_ {n+2}C^{(2m-4)}_{n -2,n +2}. \label{Vn2m exp1}
 \end{align}
Combining \eqref{rel cns 2m} with \eqref{cnn2m} gives 
\[ V_{n}^{(2m)} - V_{n-1\phantom{1}}^{(2m)} 
	 = \beta_{n} \left(V_{n+1}^{(2m-2)} + V_{n}^{(2m-2)} \right) - \beta_{n-1} \left( V_{n-1\phantom{1}}^{(2m-2)} + V_{n-2}^{(2m-2)} \right).
\]
Using the latter relation, we replace the term \(\left( \beta_{n}+\beta_ {n+1}\right) V^{(2 m -2)}_{n}\) in \eqref{Vn2m exp1} to get 
\begin{align}
 V_{n\phantom{1}}^{(2m)} &= \beta_{n} \left( V_{n+1}^{(2m-2)}+V_{n\phantom{1}}^{(2m-2)}+V_{n-1}^{(2m-2)}\right) + \beta_{n+1} V_{n-1\phantom{1}}^{(2m-2)} \nonumber\\
 &\qquad - \beta_{n+1}\beta_{n-1} \left( V_{n-1}^{(2m-4)}+V_{n-2\phantom{1}}^{(2m-4)}\right)+ \beta_{n+2}\beta_{n+1}\beta_{n}\beta_{n-1} C^{(2m-4)}_{n-2,n+2}.
\label{Vn2m exp2}
\end{align}
Consider \(n\rightarrow n-1\) and \(m\rightarrow m-1\) in the latter expression, 
\comment{\color{red} TO COMMENT just for our own checking which is 
\begin{align*}
 V_{n-1\phantom{1}}^{(2m-2)} = & \beta_{n-1} \left( V_{n}^{(2m-4)}+V_{n-1\phantom{1}}^{(2m-4)}+V_{n-2}^{(2m-4)}\right) 
 + \beta_{n} V_{n-2\phantom{1}}^{(2m-4)}
 - \beta_{n}\beta_{n-2} \left( V_{n-2}^{(2m-6)}+V_{n-3\phantom{1}}^{(2m-6)}\right)\\
 & + \beta_{n+1}\beta_{n}\beta_{n-1}\beta_{n-2} C^{(2m-6)}_{n-3,n+1},
\end{align*}}%
and replace it in \eqref{Vn2m exp2} and this yields 
\comment{\color{red} TO COMMENT just for our own checking
\begin{align*}
 V_{n\phantom{1}}^{(2m)} = & \beta_{n} \left( V_{n+1}^{(2m-2)}+V_{n\phantom{1}}^{(2m-2)}+V_{n-1}^{(2m-2)}\right) \\
 & 
 + \beta_{n+1} \Big( 
 \beta_{n-1} \left( V_{n}^{(2m-4)}+V_{n-1\phantom{1}}^{(2m-4)}+V_{n-2}^{(2m-4)}\right) 
 + \beta_{n} V_{n-2\phantom{1}}^{(2m-4)}
 - \beta_{n}\beta_{n-2} \left( V_{n-2}^{(2m-6)}+V_{n-3\phantom{1}}^{(2m-6)}\right)\\
 & + \beta_{n+1}\beta_{n}\beta_{n-1}\beta_{n-2} C^{(2m-6)}_{n-3,n+1}
 \Big)\\
 &
 - \beta_{n+1}\beta_{n-1} \left( V_{n-1}^{(2m-4)}+V_{n-2\phantom{1}}^{(2m-4)}\right)\\
 & + \beta_{n+2}\beta_{n+1}\beta_{n}\beta_{n-1} C^{(2m-4)}_{n-2,n+2}.
\end{align*}}%
\begin{align*}
 V_{n\phantom{1}}^{(2m)} &= \beta_{n} \left( V_{n+1}^{(2m-2)}+V_{n\phantom{1}}^{(2m-2)}+V_{n-1}^{(2m-2)}\right) 
 + 
 \beta_{n-1}\beta_{n+1} V_{n}^{(2m-4)}
 + \beta_{n} \beta_{n+1}V_{n-2\phantom{1}}^{(2m-4)}\\
 &\qquad
 - \beta_{n+1}\beta_{n}\beta_{n-2} \left( V_{n-2}^{(2m-6)}+V_{n-3\phantom{1}}^{(2m-6)}\right) 
 + \beta_{n+1}\beta_{n}\beta_{n-1} \left( 
 \beta_{n+1}\beta_{n-2} C^{(2m-6)}_{n-3,n+1}
 + \beta_{n+2} C^{(2m-4)}_{n-2,n+2}\right).
\end{align*}
If we replace the term \(V_{n-2\phantom{1}}^{(2m-4)}\) by the corresponding expression given by the latter relation and successively continuing the process then one can deduce the following expressions for \(V_{n\phantom{1}}^{(2m)}\) as follows 
\begin{align*}
V_{n\phantom{1}}^{(2)} &= \b_{n},\\
	V_{n\phantom{1}}^{(4)} &=V_{n\phantom{1}}^{(2)}\left(V_{n+1}^{(2)}+V_{n\phantom{1}}^{(2)}+V_{n-1}^{(2)}\right),\\ 
	V_{n\phantom{1}}^{(6)} 
 &= V_{n\phantom{1}}^{(2)}\left( V_{n+1}^{(4)}+V_{n\phantom{1}}^{(4)} + V_{n-1}^{(4)} + V_{n+1}^{(2)}V_{n-1}^{(2)} \right). 
\end{align*}
For higher orders we compute the coefficients $V_{n\phantom{1}}^{(2m)}$ recursively as stated below. We opted for not giving the expressions in terms of $\beta_n$ since those are rather long. For $m=4,5$, we have 
\begin{align*}
%
V_{n\phantom{1}}^{(8)} &= V_{n\phantom{1}}^{(2)}\left( V_{n+1}^{(6)}+V_{n\phantom{1}}^{(6)} + V_{n-1}^{(6)} \right)+V_{n\phantom{1}}^{(4)}V_{n+1}^{(2)}V_{n-1}^{(2)} +V_{n+1}^{(2)}V_{n\phantom{1}}^{(2)}V_{n-1}^{(2)}\left(V_{n+2}^{(2)}+V_{n-2}^{(2)} \right),\\
 V_{n\phantom{1}}^{(10)}
 &= V_{n\phantom{1}}^{(2)} \left(V^{(8)}_ {n+1}+V_{n\phantom{1}}^{(8)}+V^{(8)}_ {n-1}\right)+V_{n\phantom{1}}^{(6)} V^{(2)}_ {n+1} V^{(2)}_ {n-1}
 +V^{(2)}_ {n+1} V_{n\phantom{1}}^{(2)} V^{(2)}_ {n-1} \left(V^{(4)}_{n+2}+V^{(4)}_{n-2}\right)\\
 &\qquad +V^{(2)}_{n+1} V^{(2)}_{n\phantom{1}} V^{(2)}_{n-1} \left\{\left(V^{(2)}_{n\phantom{1}}+V^{(2)}_ {n-1}\right)V^{(2)}_ {n+2}	+\left(V^{(2)}_ {n+1}+V^{(2)}_{n\phantom{1}}\right) V^{(2)}_ {n-2}+V^{(2)}_{n+2} V^{(2)}_{n-2} \right\}.
\end{align*}

\comment{\color{red} Question for Peter: I tested in Maple and the formula below does not match, mainly due to the tail part (last terms)
\begin{align*}
 V_{n\phantom{1}}^{(10)} &= V_{n\phantom{1}}^{(2)} \left(V^{(8)}_ {n+1}+V_{n\phantom{1}}^{(8)}+V^{(8)}_ {n-1}\right)+V^{(2)}_ {n+1} V^{(2)}_ {n-1}V_{n\phantom{1}}^{(6)} 
 +V^{(2)}_ {n-1} V_{n\phantom{1}}^{(2)} V^{(2)}_ {n-1} \left(V_{n+2}^{(4)}+V_{n-2}^{(4)}+V_{n-2}^{(4)} \right)\\ &\qquad
 +V^{(2)}_ {n+1} V_{n\phantom{1}}^{(2)} V^{(2)}_ {n-1} \left(V^{(2)}_ {n+2} V^{(2)}_{n+3}+V^{(2)}_{n-2}V^{(2)}_{n -3} \right).
\end{align*}
}
\comment{\begin{align*}
	V_{n\phantom{1}}^{(8)} &= \beta_{n} \Big(\beta_ {n-1}^{3}+\left(3 \beta_{n}+2 \beta_ {n-2}+\beta_ {n+1}\right) \beta_ {n-1}^{2}\\
 &\qquad+\left(\beta_ {n+1}^{2}+\left(4 \beta_{n}+\beta_ {n-2}+\beta_ {n+2}\right) \beta_ {n+1}+3 \beta_{n}^{2}+2 \beta_ {n-2} \beta_{n}+\beta_ {n-2} \left(\beta_ {n-2}+\beta_{n -3}\right)\right) \beta_ {n-1}\\ 
 &\qquad +\left(\beta_ {n+1}^{2}+\left(3 \beta_{n}+2 \beta_ {n+2}\right) \beta_ {n+1}+3 \beta_{n}^{2}+2 \beta_{n} \beta_ {n+2}+\beta_ {n+2} \left(\beta_ {n+2}+\beta_{n +3}\right)\right) \beta_ {n+1}+\beta_{n}^{3}\Big),\\
\end{align*}}

\begin{remarks}
\begin{enumerate}[(i)]
\item[]
\item
For the case when $\la=-\tfrac12$, Proposition \ref{lm34} was proved by Benassi and Moro \cite{refBM20}, using a result in \cite{refBMX92}. Although it is straightforward to modify the proof presented therein for the case when $\la\not=-\tfrac12$, we hereby present an alternative approach purely depending on the structure relation of the semiclassical polynomials. 
\item {Equations such as \eqref{Vn2m} for recurrence relation coefficients are sometimes known as \textit{Laguerre-Freud equations} \cite{refFreud76,refLag}; see also \cite{refBelRon,refHHR,refMagnus86,refMagnus95,refMagnus99}. }
%
%

\item
When $m=2$ the discrete equation is
\beq 4\b_{n} \big(\b_{n-1} + \b_{n} + \b_{n+1}\big)-2t\b_{n}=n+(\la+\tfrac12)[1-(-1)^n)], \label{eq:rr42} \eeq
which is dP$_{\rm I}$
and when $m=3$ the discrete equation is
\begin{align}6\b_{n} \big(\b_{n-2} \b_{n-1} &+ \b_{n-1}^2 + 2 \b_{n-1} \b_{n} + \b_{n-1} \b_{n+1} 
+ \b_{n}^2 + 2 \b_{n}\b_{n+1} + \b_{n+1}^2 + \b_{n+1} \b_{n+2}\big) \nonumber\\
&-2t\b_{n}=n+(\la+\tfrac12)[1-(-1)^n)],\label{eq:rr62}\end{align}
which is a special case of dP$_{\rm I}^{(2)}$, the second member of the discrete \p\ I hierarchy. For further information about the discrete \p\ I hierarchy, see \cite{refCJ99a,refCJ99b}. Equations \eqref{eq:rr42} and \eqref{eq:rr62} with $t=0$ were derived by Freud \cite{refFreud76}; see also \cite{refMagnus85,refWVAbk}.
Further equation \eqref{eq:rr42} and \eqref{eq:rr62} with $\la=-\tfrac12$ are also known as ``string equations" and arise in important physical applications such as two-dimensional quantum gravity, cf.~\cite{refDS,refGMig,refFIK91,refFIK92,refKMMOZ,refPS}.
\end{enumerate}
\end{remarks}

\subsection{Asymptotics for the recurrence coefficients as {$n\to\infty$}}

In 1976, Freud \cite{refFreud76} conjectured that the asymptotic behavior of recurrence coefficients $\b_{n}$ in the recurrence relation \eqref{eq:srr} satisfied by monic polynomials ${\{P_{n} (x) \}_{n\geq0}}$ orthogonal with respect to the weight
\begin{equation*}\w(x) = |x|^{\rho}\exp(-|x|^{m}),\end{equation*}
with $x \in \R$, $\rho>-1$, $m>0$ could be described by
\begin{equation}\label{Freudconj}
\ds \lim_{n\rightarrow \infty} \frac{\b_{n}}{n^{2/m}}= \left[ \frac{\Gamma(\tfrac{1}{2}m)\, \Gamma(1+\tfrac{1}{2}m)}{\Gamma(m+1)}\right]^{2/m}.
\end{equation}
Freud stated the conjecture for orthonormal polynomials, proved it for $m=2,4,6$ and also showed that \eqref{Freudconj} is valid whenever the limit on the left-hand side exists. Magnus \cite{refMagnus85} proved Freud's conjecture for the case when $m$ is an even positive integer and also for weights 
\begin{equation*}w(x)=\exp\{-Q(x)\}, \end{equation*}
where $Q(x)$ is an even degree polynomial with positive leading coefficient. We refer the reader to \cite[\S4.18]{refNevai86} for a detailed history of solutions to Freud's conjecture up to that point. The conjecture was settled by Lubinsky, Mhaskar and Saff in \cite{refLubinskyMS} as a special case of a more general result for recursion coefficients of exponential weights, see also \cite{refLubinskyMS86}. In \cite{refLB88}, Lubinsky and Saff introduced the class of {\it very smooth Freud weights} 
of order $\a$ with conditions on $Q$ that are satisfied when $Q$ is of the form $x^{\a},$ $\a>0$. Associated with each weight in this class, one can define $a_{n}$ as the unique, positive root of the equation (cf.~\cite[p.\ 67]{refLubinskyMS} and references therein)
\begin{equation}\label{LMSnumber} n= \frac{2}{\pi} \ds \int_{0}^{1} \frac{a_{n}s\, Q'(a_{n}s)}{\sqrt{1-s^2}}\,\d{s}.\end{equation} 
 
\begin{theorem}\label{FC} Consider the generalised higher order Freud weight \eqref{freudg}. Then the recurrence coefficients $\b_{n}$ associated with this weight satisfy 
	\begin{equation*}\label{eq415}\lim_{n\to \infty}	\frac{\b_{n}(t;\la)}{n^{1/m}} = \frac 14\left(\frac{(m-1)!}{(\tfrac12)_m}\right)^{\!1/m}.
 \end{equation*}
\end{theorem}
\begin{proof}Let $Q(x)=\tfrac{1}{2}x^{2m}$, then evaluating \eqref{LMSnumber} yields 
 {$\ds n=\frac{a_{n}^{2m}(\tfrac12)_m}{(m-1)!}$}\comment{\frac{2ma_{n}^{2m}}{\pi}\ds \int_{0}^{1} \frac{s^{2m}}{\sqrt{1-s^2}}\,\d {s} 
 = \frac{ma_{n}^{2m}}{\pi}\ds \int_{0}^{1} \frac{t^{m-1/2}}{\sqrt{1-t}}\,\d{t}\\
 &=\frac{ma_{n}^{2m}}{\pi}B(m+\tfrac12,\tfrac12) 
 =\frac{ma_{n}^{2m}}{\pi}\frac{\Gamma(m+\tfrac12)\,\Gamma{(\tfrac12)}}{\Gamma(m+1)} 

where $B(p,q)$ denotes the Beta function (cf.~\cite[eq.\ 5.12.1]{DLMF}). Hence $\ds a_{n}^2=\left( \frac{(m-1)!}{(\tfrac12)_m}\, n\right)^{\!1/m}$} and the result {is a straightforward consequence of the more general result in} \cite[Theorem 2.3]{refLubinskyMS} taking $W(x)=\exp\{-Q(x)\}$, $w=|x|^{\la+1/2}$, $P(x)=\tfrac{1}{2}tx^2$ and $\Psi(x)=1$.
\end{proof}

\begin{remark}
Taking $m=2$ in Theorem \ref{FC}, we recover \cite[Corollary 4.2 (ii)]{refCJ18} for the recurrence coefficients associated with the generalised quartic Freud weight $|x|^{2\la+1}\exp(tx^2-x^4)$ which satisfy 
	\[\lim_{n\to \infty}	\frac{\b_{n}(t;\la)}{\sqrt{n}} = \frac{1}{\sqrt{12}},\]
	while, for $m=3$, the recurrence coefficients associated with the generalised sextic Freud weight $|x|^{2\la+1}\exp(tx^2-x^6)$ satisfy 
	\[\lim_{n\to \infty}\frac{\b_{n}(t;\la)}{\sqrt[3]{n}} = \frac1{\sqrt[3]{60}},\] as shown in \cite[Corollary 4.8]{{refCJ21b}}.
\end{remark}

\section{Generalised higher order Freud polynomials}\label{sec:relations}
\subsection{Differential equations}
The second order differential equations satisfied by generalised higher order Freud polynomials can be obtained by using ladder operators as was done for the special cases $m=2$ and $m=3$ in \cite[Theorem 6]{refCJK} and \cite[Theorem 4.3]{refCJ21a}, respectively. {An alternative approach is given by Maroni in \cite{refMaroni} and \cite{refMaroni2}. }

\begin{proposition}
 The polynomial sequence $\big\{P_{n}(x)\big\}_{n\geq 0}$ orthogonal with respect to the generalised higher order Freud weight \eqref{freudg} is a solution to the differential equation 
 \begin{equation*}
 J(x;n) \deriv[2]{P_{n+1}}{x}(x) + K(x;n) \deriv{P_{n+1}}{x}(x) + L(x;n) P_{n+1}(x) =0, 
 \end{equation*}
where 
\begin{align*}
 & J(x;n) = x D_{n+1}(x) , \\
 & K(x;n)= C_0(x) D_{n+1}(x) - x \deriv{D_{n+1}}{x}(x) + D_{n+1}(x), \\
 & L(x;n) = \mathcal{W}\left(\tfrac{1}{2} (C_{n+1}(x) - C_0(x) ), D_{n+1}(x) \right) - D_{n+1}(x)\sum_{j=0}^n \frac{1}{\b_j} D_j(x), 
\end{align*}
with 
\begin{align*}
 & C_{n+1} (x) = - C_{n}(x) + \frac{2x}{\b_{n}} D_{n}(x), \qquad
 D_{n+1}(x) = -x + \frac{\b_{n}}{\b_{n-1}} D_{n-1}(x) + \frac{x^2}{\b_{n}} D_{n}(x) - x C_{n}(x), 
\end{align*}
subject to the initial conditions $C_0(x) = -1 + 2(tx^2-mx^{2m}+\la+1)$, $D_{-1}(x) = 0$ and 
\[D_0(x) = 2x \left\{m \sum_{j=1}^{m} \mu_{2j-2}(t,\la) x^{2m-2j}
 -t \mu_0(t,\la)\right\}.
\]
\end{proposition}

\subsection{Mixed recurrence relations}
{We first consider the connection formula between the corresponding sequences of generalised higher order Freud orthogonal polynomials in the framework of Christoffel transformations when the measure is modified by multiplying with a polynomial. In our case, the measure is modified by a quadratic factor.}
\begin{theorem}\label{mrec} Let ${\big\{P_{n}(x;\la)\big\}_{n\geq0}}$ be the sequence of monic generalised higher order Freud polynomials orthogonal with respect to the weight \eqref{freudg}, then, for $m,n$ fixed, 
	\begin{subequations}
		\begin{align}\label{rreven}
xP_{2n}(x;\la+1)&=P_{2n+1}(x;\la), \\
\label{rrodd}
x^2P_{2n-1}(x;\la+1)&=xP_{2n}(x;\la)-\left\{\b_{2n}(\la)+\ds \frac{P_{2n+1}'(0;\la)}{P_{2n-1}'(0;\la)}\right\}P_{2n-1}(x;\la).
		\end{align} 
	\end{subequations} 
\end{theorem}

\begin{proof}Let $P_{n}(x;\la+1)$ be the polynomials associated with the even weight function 
	\begin{align*}\w(x;\la+1)&=|x|^{2\la+3}\exp\big(tx^2-x^{2m}\big) 
		=x^2\w(x;\la), \quad m=2,3,\dots\ .\end{align*} 
	The factor $x^2$ by which the weight $\w(x;\la)$ is modified has a double zero at the origin and therefore Christoffel's formula (cf.~\cite[Theorem 2.5]{refSzego},~\cite[Theorem 2.7.1]{Ismail}), applied to the monic polynomials $P_{n}(x;\la+1)$, is
	\[
	x^2P_{n}(x;\la+1)=\frac{1}{P_{n}(0;\la)P_{n+1}'(0;\la)-P_{n}'(0;\la)P_{n+1}(0;\la)}\left|\begin{matrix} P_{n}(x;\la) & P_{n+1}(x;\la) & P_{n+2}(x;\la)\\
		P_{n}(0;\la) & P_{n+1}(0;\la) & P_{n+2}(0;\la)\\
		P_{n}'(0;\la) & P_{n+1}'(0;\la) & P_{n+2}'(0;\la)\\\end{matrix}\right|.\]
	Since the weight $\w(x;\la)$ is even, we have that $P_{2n+1}(0;\la)=P_{2n}'(0;\la)$ while $P_{2n}(0;\la)\neq0$ and $P_{2n+1}'(0;\la)\neq0$, hence 
	\[
	x^2P_{n}(x;\la+1)=\frac{-1}{ P_{n}'(0;\la)P_{n+1}(0;\la)}\left|\begin{matrix} P_{n}(x;\la) & P_{n+1}(x;\la) & P_{n+2}(x;\la)\\
		0 & P_{n+1}(0;\la) &0\\
		P_{n}'(0;\la) & 0 & P_{n+2}'(0;\la)\\\end{matrix}\right|,
	\] for $n$ odd, while, for $n$ even, 
	\[
	x^2P_{n}(x;\la+1)=\frac{1}{P_{n}(0;\la)P_{n+1}'(0;\la) }\left|\begin{matrix} P_{n}(x;\la) & P_{n+1}(x;\la) & P_{n+2}(x;\la)\\
		P_{n}(0;\la) & 0& P_{n+2}(0;\la)\\
		0 & P_{n+1}'(0;\la) &0\\\end{matrix}\right|.
	\]
	This yields
	\beq\label{even}
	x^2P_{n}(x;\la+1)=P_{n+2}(x;\la)-a_{n}P_{n}(x;\la),
	\eeq where
	\beq \nonumber a_{n}=\begin{cases} \ds\frac{P_{n+2}(0;\la)}{P_{n}(0;\la)},\quad &\text{for}\quad n\quad\text{even},\\[6pt]
		\ds\frac{P_{n+2}'(0;\la)}{P_{n}'(0;\la)},\quad &\text{for}\quad n\quad\text{odd}.\end{cases}\label{ed}\eeq Using the three-term recurrence relation \eqref{eq:srr} to eliminate $P_{n+2}(x;\la)$ in \eqref{even}, we obtain
	\begin{align*}
		x^2P_{n}(x;\la+1)=xP_{n+1}(x;\la)-(\b_{n+1}(\la)+a_{n})P_{n}(x;\la).
	\end{align*} It follows from \eqref{bodd} that, for $n$ even, $\b_{n+1}(\la)+a_{n}=0$ and the result follows.	
\end{proof}	
 
\begin{theorem}\label{qoeq}
	For a fixed $m=2,3,\dots$, let ${\big\{P_{n}(x;\la)\big\}_{n\geq0}}$ be the sequence of monic generalised higher order Freud polynomials orthogonal with respect to the weight \eqref{freudg}. Then, for $n$ fixed, 
	\begin{subequations}
		\begin{align}\label{qorodd}	
P_{2n+1}(x;\la)&=P_{2n+1}(x;\la+1)+\b_{2n}(\la+1)P_{2n-1}(x;\la+1),\\\label{qoreven}
P_{2n}(x;\la)&=P_{2n}(x;\la+1)-\frac{\b_{2n}(\la)\b_{2n-1}(\la+1)P_{2n-1}'(0;\la)}{P_{2n+1}'(0;\la)}P_{2n-2}(x;\la+1).
		\end{align}
	\end{subequations}
\end{theorem} 
\begin{proof}
	Substitute \eqref{rreven} into the three term recurrence relation \begin{equation}\label{nrr}P_{2n+1}(x;\la)=xP_{2n}(x;\la)-\b_{2n}(\la)P_{2n-1}(x;\la),\end{equation} to eliminate $P_{2n+1}(x;\la)$ and obtain
	\[xP_{2n}(x;\la+1)=xP_{2n}(x;\la)-\b_{2n}(\la)P_{2n-1}(x;\la).\] Let $\ds a_{2n}=\frac{P_{2n+1}'(0;\la)}{P_{2n-1}'(0;\la)}$. Substitute \eqref{rrodd} into \eqref{nrr} to eliminate $P_{2n-1}(x;\la)$ and obtain
	\begin{align}\label{nn}xP_{2n}(x;\la+1)&=xP_{2n}(x;\la)- \frac{\b_{2n}(\la)}{\b_{2n}(\la)+a_{2n}}\left(xP_{2n}(x;\la)-x^2P_{2n-1}(x;\la+1)\right)
		.\end{align} Simplification and rearrangement of terms in \eqref{nn} yields \[\left(1-\frac{\b_{2n}(\la)}{\b_{2n}(\la)+a_{2n}}\right)P_{2n}(x;\la)	= P_{2n}(x;\la+1)-\frac{\b_{2n}(\la)}{\b_{2n}(\la)+a_{2n}}xP_{2n-1}(x;\la+1),\] then, using the three term recurrence relation to eliminate $xP_{2n-1}(x;\la+1)$, we obtain
	\begin{align*}\left(1-\frac{\b_{2n}(\la)}{\b_{2n}(\la)+a_{2n}}\right)P_{2n}(x;\la)&= \left(1-\frac{\b_{2n}(\la)}{\b_{2n}(\la)+a_{2n}}\right)P_{2n}(x;\la+1)\\& \qquad+\frac{\b_{2n}(\la)}{\b_{2n}(\la)+a_{2n}}\b_{2n-1}(\la+1)P_{2n-2}(x;\la+1),\end{align*} which simplifies to \eqref{qoreven}.
	Substituting \eqref{rreven} into the three term recurrence relation \[P_{2n+1}(x;\la+1)=xP_{2n}(x;\la+1)-\b_{2n}(\la+1)P_{2n-1}(x;\la+1),\] yields \eqref{qorodd}. 
\end{proof} 
{Theorem \ref{qoeq} gives the connection formula between the corresponding sequences of generalised higher order Freud polynomials in the framework of Geronimus transformations, the inverse of a Christoffel transformation. For more on quadratic Geronimus transformations of a weight $\w(x)$, where $(x^2-c)v(x) = \w(x)$, see \cite{refGMVH}. The generalised Christoffel formula, where the weight is modified by a rational function, often referred to as an Uvarov transformation, can also be considered as the Darboux transformation of an integrable system (cf. \cite{refBM,Ismail}) and are considered in the framework of Gaussian quadrature rules in \cite{refGautschi1,refGautschi2}.}  
\subsection{Quasi-orthogonality for {$\lambda\in(-2,-1)$}}
Theorem \ref{qoeq} yields the quasi-orthogonality of generalised higher order Freud polynomials for $-2<\lambda<-1$.
\begin{theorem} Suppose $-2<\la<-1$. For each fixed $m=2,3,\dots$, the generalised higher order Freud polynomial $P_{n}(x;\la)$ is quasi-orthogonal of order $2$ on $\R$ with respect to the weight \beq\nonumber |x|^{2\la+3}\exp\big(tx^2-x^{2m}\big),\qquad t \in\R.\eeq 
\end{theorem}
\begin{proof}
	Suppose $-2<\la<-1$, then $\la+1>-1$.
	When $n$ is even, we have from \eqref{qoreven} that \begin{align}\nonumber\imp &x^kP_{n}(x;\la)\,\,|x|^{2\la+3}\exp\big(tx^2-x^{2m}\big)\,\d x \\&\label{oreven}= \imp x^kP_{n}(x;\la+1)\,\,|x|^{2\la+3}\exp\big(tx^2-x^{2m}\big)\,\d x\\&\qquad-\frac{\b_{n}(\la)\b_{n-1}(\la+1)P_{n-1}'(0;\la)}{P_{n+1}'(0;\la)}\imp x^kP_{n-2}(x;\la+1)\,\,|x|^{2\la+3}\exp\big(tx^2-x^{2m}\big)\,\d x,	\nonumber 
	\end{align} 
	while, for $n$ is odd, it follows from \eqref{qorodd} that\begin{align}
		\nonumber\imp &x^kP_{n}(x;\la)\,\,|x|^{2\la+3}\exp\big(tx^2-x^{2m}\big)\,\d x \\&\label{orodd}= \imp x^kP_{n}(x;\la+1)\,\,|x|^{2\la+3}\exp\big(tx^2-x^{2m}\big)\,\d x\\&\qquad+\b_{n}(\la+1)\imp x^kP_{n-2}(x;\la+1)\,\,|x|^{2\la+3}\exp\big(tx^2-x^{2m}\big)\,\d x.	\nonumber
	\end{align} Since $\la+1>-1$, it follows from the orthogonality of the generalised higher order Freud polynomials that \[\imp x^kP_{n}(x;\la+1)\,\,|x|^{2\la+3}\exp\big(tx^2-x^{2m}\big)=0\qquad\text{for}\quad k=0,\dots,n-1,\] and we see that all the integrals on the righthand side of \eqref{oreven} and \eqref{orodd} are equal to zero for $k=0,\dots,n-3$. 
\end{proof}
\section{Zeros of generalised higher order Freud polynomials}\label{sec:zeros} 

\subsection{Asymptotic zero distribution} 
The asymptotic behaviour of the recurrence coefficients of generalised higher order Freud polynomials orthogonal with respect to \eqref{freudg}, satisfying Freud's conjecture, given by \eqref{Freudconj}, is independent of the values of $t$ and $\la$. The asymptotic behaviour implies that the recurrence coefficients are regularly varying, irrespective of $t$ and $\la$. To consider the asymptotic distribution of the zeros of generalised higher order Freud polynomials orthogonal with respect to the weight \eqref{freudg} as $n\to \infty$, we use an appropriate scaling and apply the property of regular variation as detailed in \cite{refKVA1999}. 
\begin{theorem}\label{realasymp} Let $\phi(n)=n^{1/(2m)}$ and assume that $n,N$ tend to infinity in such a way that the ratio $n/N\to \ell$. Then, for the sequence of scaled monic polynomials $P_{n,N}(x)=(\phi(N))^{-n}P_{n}(\phi(N)x)$ associated with the generalised higher order Freud weight \eqref{freudg}, the asymptotic zero distribution, as $n\to \infty$, has density \begin{equation}\label{dens}a_m(\ell)=\frac{2m}{c\pi (2m-1)}\left(1-{x^2}/{c^2}\right)^{\!1/2}{}_2F_1\left(1,1-m;\tfrac32-m;{x^2}/{c^2}\right),\end{equation} where \begin{equation*}c=2a\ell^{1/(2m)} \quad \text{with} \quad a= \tfrac12\left(\frac{(m-1)!}{\!(\tfrac12)_m}\right)^{\!1/(2m)},\end{equation*} defined on the interval 
$(-2 a \ell^{1/(2m)},2a\ell^{1/(2m)})$.\end{theorem}
\begin{proof} The scaled monic polynomials $P_{n,N}(x)=(\phi(N))^{-n}P_{n}(\phi(N)x)$ associated with the generalised higher order Freud weight \eqref{freudg} have recurrence coefficient $\b_{n,N}(t;\la)=\frac{\b_{n}(t;\la)}{\left(\phi(N)\right)^2}$. Since $\phi:\R^+\to \R^+$ and, for every $\ell>0$, we have \[\lim_{x\to\infty}\frac{\phi(x\ell)}{\phi(x)}=\ell^{1/(2m)},\] $\phi$ is regularly varying at infinity with exponent of variation $\tfrac{1}{2m}$ (cf.~\cite{refVAG1989}). Since it follows from \eqref{eq415} that \[\lim_{n/N\to \ell}\sqrt{\b_{n,N}(t;\la)}=\lim_{n/N\to\ell}\frac{\sqrt{\b_{n}(t;\la)}}{\phi(n)}\frac{\phi(n)}{\phi(N)}=a\ell^{1/(2m)},\] the recurrence coefficients $\b_{n.N}(t;\la)$ are said to be regularly varying at infinity with index $\tfrac{1}{2m}$ (cf.~\cite[Section 4.5]{refKVA1999}). From the property of regular variation, using \cite[Theorem 1.4]{refKVA1999}, it follows that the asymptotic zero distribution has density \begin{align*}\frac {1}{\pi \ell}\int_0^{\ell}s^{-1/(2m)}&\left(2a-x s^{-1/(2m)}\right)^{\!-1/2}\left(2a+x s^{-1/(2m)}\right)^{\!-1/2}\,\d s\label{KVA}&\\&=\frac{m}{a\pi \ell}\int_0^{\ell^{1/(2m)}}y^{2m-2}\left(1-\left(\frac{x}{2ay}\right)^{\!2} \right)^{\!-1/2}\,\d y\nonumber\\&=\frac{m}{a\pi\ell}\int_0^{\ell^{1/(2m)}}y^{2m-2}\sum_{k=0}^{\infty}\frac{\left(\tfrac12\right)_k}{k!}\left(\frac{x}{2a y}\right)^{\!2k}\,\d y\nonumber
\\&=\frac{m}{a\pi\ell}\sum_{k=0}^{\infty}\frac{\left(\tfrac12\right)_k}{k!}\left(\frac{x}{2a}\right)^{\!2k}\int_0^{\ell^{1/(2m)}}y^{2m-2k-2}\,\d y\nonumber
\\&=\frac{m}{a\pi\ell^{1/(2m)}}\sum_{k=0}^{\infty}\frac{\left(\tfrac12\right)_k}{k!}\frac{1}{2m-2k-1}\left(\frac{x}{2a\ell^{1/(2m)}}\right)^{\!2k}\nonumber\\&=\frac{m}{a\pi\ell^{1/(2m)}(2m-1)}\sum_{k=0}^{\infty}\frac{\left(\tfrac12\right)_k\left(\tfrac12-m\right)_k}{\left(\tfrac 32-m\right)_k k!}\left(\frac{x}{2a\ell^{1/(2m)}}\right)^{\!2k}\nonumber
\\&=\frac{m}{a\pi(2m-1)\ell^{1/(2m)}}{}_2F_1\left(\tfrac12,\tfrac12-m;\tfrac32-m;\left(\frac{x}{2a\ell^{1/(2m)}}\right)^{\!2}\right).
\end{align*} 
\end{proof}

Figure \ref{density} shows the zeros and the asymptotic distribution according to Theorem \ref{realasymp}.

	\begin{figure}[ht]
	\begin{center}
\includegraphics[width=3in]{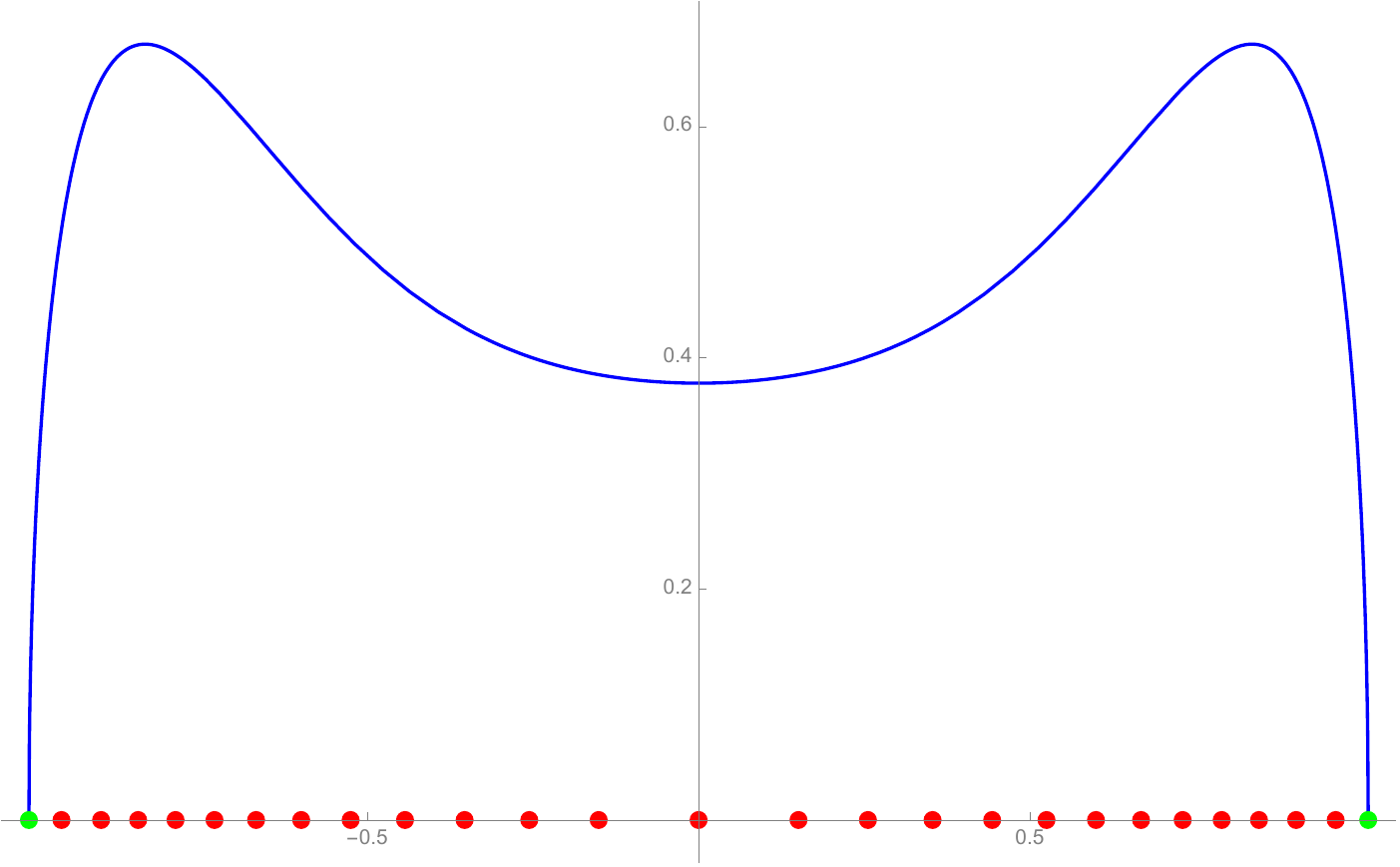}
	\caption{\label{density}The zeros of $P_{n,N}(x)$ ({\color{red}{red}}) for $\la =0.5$, $t=1$, $m=3$, $n=N=10$ and $\ell=1$ with the corresponding limiting distribution (\ref{dens}) ({\color{dkb}{blue}}) and endpoints $(-2a,0)$ and $(2a ,0)$ ({\color{dkg}{green}}).}
		\end{center}
	\end{figure}
	
	Figure \ref{figure2} shows the asymptotic distribution of zeros according to Theorem \ref{realasymp} for various values of $\ell$.

	\begin{figure}[ht]
	\begin{center}
\includegraphics[width=3in]{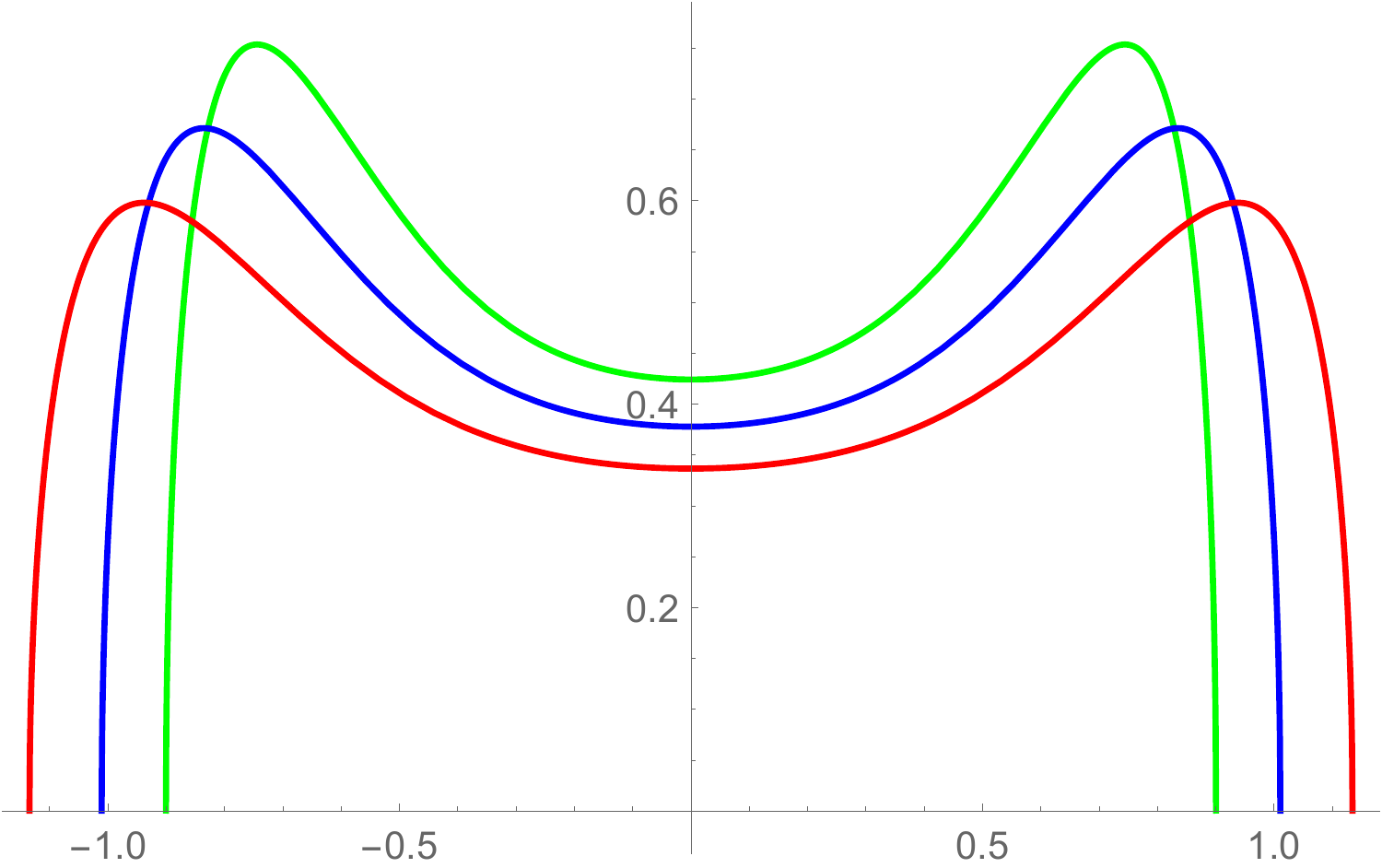} 
	\caption{\label{figure2}The limiting distribution of the zeros $a_3(\ell)$ for $\ell =0.5$ ({\color{dkg}{green}}), $\ell=1$ ({\color{dkb}{blue}}) and $\ell=2$ ({\color{red}{red}}).}
		\end{center}
	\end{figure}

\remark Note that the formula on \cite[p. 189, line 22]{refKVA1999} should be $\ds\frac 1t\int_0^t\frac{1}{s^{\la}}\w'_{[b-2a,b+2a]}(xs^{-\la})\,\d s$.
\subsection{Bounds for the extreme zeros} From the three-term recurrence relation \eqref{eq:3rr}, we obtain bounds for the extreme zeros of monic generalised higher order Freud polynomials.

\begin{theorem} For each $n=2,3,\dots,$ the largest zero, $x_{1,n}$, of monic generalised higher order Freud polynomials $P_{n}(x)$ orthogonal with respect to the weight \eqref{freudg}, satisfies 
\[0<x_{1,n} <\max_{1\leq k\leq n-1}\sqrt{c_{n}\b_k(t;\la)},\]
where $\ds c_{n}=4\cos^2\left(\frac{\pi}{n+1}\right)+\ep$, $\ep>0$.

\end{theorem}
\begin{proof} 
The upper bound for the largest zero $x_{1,n}$ follows by applying \cite[Theorem 2 and 3]{refIsmailLi}, based on the Wall-Wetzel Theorem to the three-term recurrence relation \eqref{eq:3rr}. 
\end{proof}

\subsection{Monotonicity of the zeros}

\begin{theorem}
\label{mono1} Consider $0<x_{\lfloor n/2 \rfloor,n}<\dots<x_{2,n} <x_{1,n} $, the positive zeros of monic orthogonal polynomials $P_{n}(x)$ with respect to the generalised higher order Freud weight \eqref{freudg} where $\lfloor k \rfloor$ denotes the largest integer less than or equal to $k$. 
Then, for $\la>-1$, $t\in \R$ and for a fixed value of $\nu$, $\nu\in \{1,2,\dots, \lfloor n/2 \rfloor\}$, the $\nu$-th zero $x_{n,\nu} $ increases when (i), $\la$ increases; and (ii),
$t$ increases.

\end{theorem} 
\begin{proof} This follows from \cite[Lemma 4.5]{refCJ21a}, taking $C(x)=x$, $D(x)=x^2$, $\rho=2\la+1$ and $\w_0(x)=\exp(-x^{2m})$.
\end{proof}

\subsection{Interlacing of the zeros}
Next, for fixed $\la>-1$, $t\in \R$ and $k\in (0,1]$, we consider the relative positioning of the zeros of the monic generalised higher order Freud polynomials $\big\{P_{n}(x;\la)\big\}$ orthogonal with respect to the weight \eqref{freudg}, and the zeros of $\{P_{n}(x;\la+k)$, $k\in (0,1]$, orthogonal with respect to the weight \beq \nonumber
\w(x;t,\la)=|x|^{2\la+2k+1}\exp\left(tx^2-x^{2m}\right), \qquad m=2,3,\dots\ .\eeq 
The zeros of monic generalised higher order Freud polynomials $\big\{P_{n}(x;\la)\big\}$ orthogonal with respect to the symmetric weight
\eqref{freudg} are symmetric around the origin. 
We denote the positive zeros of $P_{2n} (x;\la)$ by \[0<x_{n,2n}^{\la}<x_{n-1,2n}^{\la}<\dots<x_{2,2n}^{\la}<x_{1,2n}^{\la},\] and the positive zeros of $P_{2n+1} (x;\la)$ by \[0<x_{n,2n+1}^{\la}<x_{n-1,2n+1}^{\la}<\dots<x_{2,2n+1}^{\la}<x_{1,2n+1}^{\la},\] {noting that $x_{n+1,2n+1}=0$.}
\begin{theorem} \label{int} Let $\la>-1$ and $t\in \R$. Let $\big\{P_{n}(x;\la)\big\}$ be the monic generalised higher order Freud polynomials orthogonal with respect to the weight
\eqref{freudg}.
Then, for $\ell\in\{1,\dots,n-1\}$ and $k\in(0,1{\color{blue})}$, we have
\begin{eqnarray}\label{inteven}
x_{\ell+1,2n}^{\la}<x_{\ell,2n-1}^{\la}<x_{\ell,2n-1}^{\la+k}<x_{\ell,2n-1}^{\la+1}<x_{\ell,2n}^{\la},
\end{eqnarray}
and, for $\ell\in\{1,,\dots,n\}$,
\begin{eqnarray}\label{intodd}x_{\ell+1,2n+1}^{\la}<x_{\ell,2n}^{\la}<x_{\ell,2n}^{\la+k}<x_{\ell,2n}^{\la+1}=x_{\ell,2n+1}^{\la}.
\end{eqnarray}\end{theorem}

\begin{proof} The zeros of two consecutive polynomials in the sequence of generalised higher order Freud orthogonal polynomials are interlacing, that is, 
\beq\label{3.13a}
0<x_{n,2n}^{\la}<x_{n-1,2n-1}^{\la}<x_{n-1,2n}^{\la}<\dots<x_{2,2n}^{\la}<x_{1,2n-1}^{\la}<x_{1,2n}^{\la}
\eeq and 
\beq\label{3.13b}
0<x_{n,2n}^{\la}<x_{n,2n+1}^{\la}<x_{n-1,2n}^{\la}<\dots<x_{2,2n+1}^{\la}<x_{1,2n}^{\la}<x_{1,2n+1}^{\la}.
\eeq On the other hand, we proved in Theorem \ref{mono1} that the positive zeros of generalised higher order Freud polynomials monotonically increase as the parameter$\la$ increases. This implies that, for each fixed $\ell\in\{1,2,\dots,n\}$ and $k\in(0,1)$, 
\beq\label{moe} x_{ \ell ,2n}^{\la}<x_{\ell,2n}^{\la+k}<x_{\ell,2n }^{\la+1},\eeq and
\beq\label{moo} x_{ \ell,2n-1}^{\la}<x_{\ell,2n-1}^{\la+k}<x_{\ell,2n-1}^{\la+1}.\eeq 
Next, we prove that the zeros of $P_{2n}(x;\la)$ interlace with those of $P_{2n-1}(x;\la+1)$. From \eqref{rrodd},
\beq\label{l+22}
P_{2n-1}(x;\la+1)=\frac{xP_{2n}(x;\la)-(\b_{2n}(\la)+P_{2n+1}'(0;\la)/P_{2n-1}'(0;\la))P_{2n-1}(x;\la)}{x^2}.
\eeq
Evaluating \eqref{l+22} at consecutive zeros $x_{\ell}=x_{\ell,n}^{(\la)}$ and $x_{\ell+1}=x_{\ell+1,n}^{(\la)}$, $\ell=1, 2, \ldots, n-1$, of $P_{2n}(x;\la)$, we obtain
\[\begin{split}
P_{2n-1}&(x_{\ell};\la+1)P_{2n-1}(x_{\ell+1};\la+1)\\&=\frac{(\b_{2n}(\la)+P_{2n+1}'(0;\la)/P_{2n-1}'(0;\la) )^{\!2}P_{2n-1}(x_{\ell};\la)P_{2n-1}(x_{\ell+1};\la)}{x_{\ell}^2 x_{\ell+1}^2}<0,
\end{split}\]
since the zeros of $P_{2n}(x;\la)$ and $P_{2n-1}(x;\la)$ separate each other. So there is at least one positive zero of $P_{2n}(x;\la+1)$ in the interval $(x_{\ell}, x_{\ell}+1)$ for each $\ell=1, 2, \ldots, n-1$
since there are exactly $n$ positive zeros and 
and this implies that
\beq
0<x_{n,2n}^{\la}<x_{n-1,2n-1}^{\la+1}<x_{n-1,2n}^{\la}<x_{n-2,2n-1}^{\la+1}< 
\dots<x_{2,2n-1}^{\la+1}<x_{2,2n}^{\la}<x_{1,2n-1}^{\la+1}<x_{1,2n}^{\la}\label{3.14b}.
\eeq
Equations \eqref{3.13a}, \eqref{moo} and \eqref{3.14b} yield \eqref{inteven}. To prove \eqref{intodd}, we note that by \eqref{rreven} the $n$ positive zeros of $P_{2n}(x,\la+1)$ and $P_{2n+1}(x;\la)$ coincide, i.e $x_{\ell,2n}^{\la+1}=x_{\ell,2n+1}^{\la}$ for $\ell\in\{1,2,\dots,n\}$, and the result follows using \eqref{3.13b} and \eqref{moe}.\end{proof}

Figure \ref{inte} shows the interlacing of the zeros of polynomials orthogonal with respect to the generalised higher order Freud weight \eqref{freudg} for $m=3$ as described in \eqref{3.13a} of Theorem \ref{int}.

	\begin{figure}[ht]
	\begin{center}
\includegraphics[width=4in]{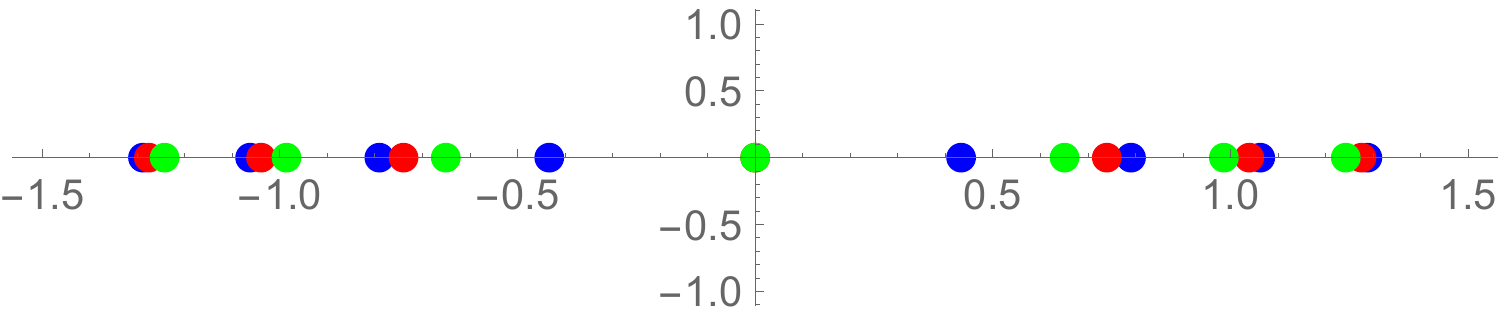} 
	\caption{\label{inte}The zeros of $P_{7}(x;\la)$ ({\color{dkg}{green}}), $P_{7}(x;\la+1)$ ({\color{red}{red}}) and $P_{8}(x;\la)$ ({\color{dkb}{blue}}) for $\la =0.5$ and $t=1$.}
		\end{center}
	\end{figure}

Figure \ref{into} illustrates the interlacing of the zeros of polynomials orthogonal with respect to the generalised higher order Freud weight \eqref{freudg} for $m=3$ as described in \eqref{3.13b} of Theorem \ref{int}.

	\begin{figure}[ht]
	\begin{center}
\includegraphics[width=4in]{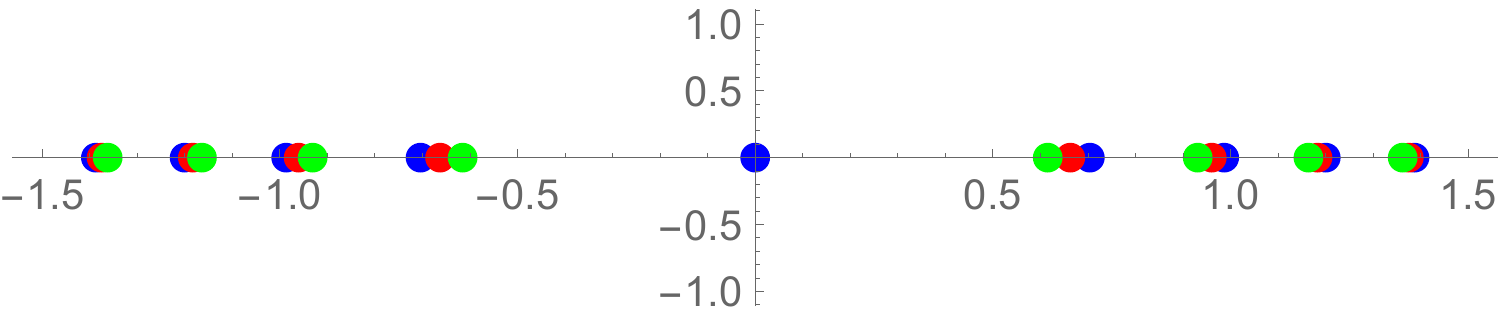} 
	\caption{\label{into}The zeros of $P_{8}(x;\la)$ ({\color{dkg}{green}}), $P_{8}(x;\la+0.5)$ ({\color{red}{red}}), $xP_8(x;\la+1)$ ({\color{dkb}{blue}}) and $P_{9}(x;\la)$ ({\color{dkb}{blue}}) for $\la =1.5$ and $t=2.3$.}
		\end{center}
	\end{figure}
\section{Quadratic decomposition of the generalised higher order Freud weight}\label{sec:qdecomp}

{We apply known results \cite[Chapter 1, Theorem 9.1]{refChihara} on the quadratic decomposition of any symmetric polynomials to this particular case of generalised higher-order Freud weights}. Precisely, if
\[
	P_{2n} (x;t,\la) = B_{n}(x^2;t,\la,) \qquad 
	P_{2n+1} (x;t,\la) = xR_{n}(x^2;t,\la), \qquad \text{for all}\quad n\geq 0, 
\] 
then from the recurrence relation \eqref{eq:srr} we have
\begin{align*}
	B_{n+1} (x;t,\la) &= R_{n+1}(x;t,\la) + \b_{2n+2} R_{n}(x;t,\la), \\
 	x R_{n} (x;t,\la) &= B_{n+1}(x;t,\la) + \b_{2n+1} B_{n}(x;t,\la) ,
\end{align*}
and this gives second order recurrence relations for both $\{B_{n}\}_{n\geq0}$ and $\{R_{n}\}_{n\geq0}$ as follows 
\begin{align*}
	& \begin{cases}
	B_{n+1}(x;t,\la) = \left(x - \b_{2n}-\b_{2n+1}\right) B_{n}(x;t,\la) - \b_{2n-1}\b_{2n} B_{n-1}(x;t,\la),\ n\geq 1, \\
	B_1(x;t,\la) = x -\b_1, \ B_0(x;t,\la)=1, 
	\end{cases}
 \\
	& \begin{cases}
	 R_{n+1}(x;t,\la) = \left(x - \b_{2n+2}-\b_{2n+1}\right) R_{n}(x;t,\la) - \b_{2n+1}\b_{2n} R_{n-1}(x;t,\la), \\
	 R_1(x;t,\la) = x -\b_1-\b_2, \ R_0(x;t,\la)=1.
	\end{cases}
\end{align*}
Furthermore, $\{B_{n}\}_{n\geq 0}$ and $\{R_{n}\}_{n\geq 0}$ satisfy the orthogonality relations 
\begin{align*}
&	\int_0^{\infty} B_k(x;t,\la) B_{n}(x;t,\la) \, x^{\lambda} \exp(tx-x^{m})\, \d x = h_{2n}(t,\la) \delta_{n,k},\\
&	\int_0^{\infty} R_k(x;t,\la) R_{n}(x;t,\la) \, x^{\lambda+1} \exp(tx-x^{m})\, \d x = h_{2n+1}(t,\la) \delta_{n,k},\quad n,k\geq 0. 
\end{align*}

\section*{Acknowledgement}
{PAC and KJ gratefully acknowledge the support of a Royal Society Newton Advanced Fellowship NAF$\backslash$R2$\backslash$180669.}
{We also thank the reviewers for helpful suggestions and additional references.}

	\def\ams{American Mathematical Society}
	\def\AAM{Acta Appl. Math.}
	\def\ARMA{Arch. Rat. Mech. Anal.}
	\def\bull{Acad. Roy. Belg. Bull. Cl. Sc. (5)}
	\def\AC{Acta Crystrallogr.}
	\def\AM{Acta Metall.}
	\def\ampa{Ann. Mat. Pura Appl. (IV)}
	\def\AP{Ann. Phys., Lpz.}
	\def\APNY{Ann. Phys., NY}
	\def\APP{Ann. Phys., Paris}
	\def\BAMS{Bull. Amer. Math. Soc.}
	\def\CJP{Can. J. Phys.}
	\def\cmp{Commun. Math. Phys.}
	\def\CMP{Commun. Math. Phys.}
	\def\cpam{Commun. Pure Appl. Math.}
	\def\CPAM{Commun. Pure Appl. Math.}
	\def\CQG{Classical Quantum Grav.}
	\def\crp{C.R. Acad. Sc. Paris}
	\def\CSF{Chaos, Solitons \&\ Fractals}
	\def\DE{Diff. Eqns.}
	\def\DU{Diff. Urav.}
	\def\ejam{Europ. J. Appl. Math.}
	\def\EJAM{Europ. J. Appl. Math.}
	\def\funk{Funkcial. Ekvac.}
	\def\FUNK{Funkcial. Ekvac.}
	\def\IP{Inverse Problems}
	\def\JAMS{J. Amer. Math. Soc.}
	\def\JAP{J. Appl. Phys.}
	\def\JCP{J. Chem. Phys.}
	\def\JDE{J. Diff. Eqns.}
	\def\JFM{J. Fluid Mech.}
	\def\JJAP{Japan J. Appl. Phys.}
	\def\JP{J. Physique}
	\def\JPhCh{J. Phys. Chem.}
	\def\JMAA{J. Math. Anal. Appl.}
	\def\JMMM{J. Magn. Magn. Mater.}
	\def\JMP{J. Math. Phys.}
	\def\jmp{J. Math. Phys}
	\def\JNMP{J. Nonl. Math. Phys.}
	\def\jpa{J. Phys. A}
	\def\JPA{J. Phys. A}
	\def\JPB{J. Phys. B: At. Mol. Phys.} 
	\def\jpb{J. Phys. B: At. Mol. Opt. Phys.} 
	\def\JPC{J. Phys. C: Solid State Phys.} 
	\def\JPCM{J. Phys: Condensed Matter} 
	\def\JPD{J. Phys. D: Appl. Phys.}
	\def\JPE{J. Phys. E: Sci. Instrum.}
	\def\JPF{J. Phys. F: Metal Phys.}
	\def\JPG{J. Phys. G: Nucl. Phys.} 
	\def\jpg{J. Phys. G: Nucl. Part. Phys.} 
	\def\JSP{J. Stat. Phys.}
	\def\JOSA{J. Opt. Soc. Am.}
	\def\JPSJ{J. Phys. Soc. Japan}
	\def\JQSRT{J. Quant. Spectrosc. Radiat. Transfer}
	\def\LMP{Lett. Math. Phys.}
	\def\LNC{Lett. Nuovo Cim.}
	\def\NC{Nuovo Cim.}
	\def\NIM{Nucl. Instrum. Methods}
	\def\NL{Nonlinearity}
	\def\NMJ{Nagoya Math. J.}
	\def\NP{Nucl. Phys.}
	\def\pl{Phys. Lett.}
	\def\PL{Phys. Lett.}
	\def\PMB{Phys. Med. Biol.}
	\def\PR{Phys. Rev.}
	\def\PRL{Phys. Rev. Lett.}
	\def\PRS{Proc. R. Soc.}
	\def\prsl{Proc. R. Soc. Lond. A}
	\def\PRSL{Proc. R. Soc. Lond. A}
	\def\PS{Phys. Scr.}
	\def\PSS{Phys. Status Solidi}
	\def\PTRS{Phil. Trans. R. Soc.}
	\def\RMP{Rev. Mod. Phys.}
	\def\RPP{Rep. Prog. Phys.}
	\def\RSI{Rev. Sci. Instrum.}
	\def\SAM{Stud. Appl. Math.}
	\def\sam{Stud. Appl. Math.}
	\def\SSC{Solid State Commun.}
	\def\SST{Semicond. Sci. Technol.}
	\def\SUST{Supercond. Sci. Technol.}
	\def\ZP{Z. Phys.}
	\def\JCAM{J. Comput. Appl. Math.}
	
	\def\OUP{Oxford University Press}
	\def\CUP{Cambridge University Press}
	\def\AMS{American Mathematical Society}

	\def\bitem{\vspace{-5pt}\bibitem}
	\def\refjl#1#2#3#4#5#6#7{\bitem{#1}\textrm{\frenchspacing#2}, \textrm{#3},
\textit{\frenchspacing#4}, \textbf{#5}\ (#7)\ #6.}
	
	\def\refjnl#1#2#3#4#5#6#7{\bitem{#1}{\frenchspacing\rm#2}, #3, 
{\frenchspacing\it#4}, {\bf#5} (#6) #7.}
	
	\def\refpp#1#2#3#4{\bitem{#1} \textrm{\frenchspacing#2}, \textrm{#3}, #4.}
	
	\def\refjltoap#1#2#3#4#5#6#7{\bitem{#1} \textrm{\frenchspacing#2}, \textrm{#3},
\textit{\frenchspacing#4} (#7). 
#6.}
	
	\def\refbk#1#2#3#4#5{\bitem{#1} \textrm{\frenchspacing#2}, \textit{#3}, #4, #5.}
	
	\def\refcf#1#2#3#4#5#6#7{\bitem{#1} \textrm{\frenchspacing#2}, \textrm{#3},
in: \textit{#4}, {\frenchspacing#5}, pp.\ #7, #6.}

\end{document}